\documentclass[12pt,a4paper,reqno,english]{amsart}

\usepackage[hypertex,
   pdftitle={On perturbed substochastic semigroups in  ordered Banach spaces with additive norm.},
   pdfauthor={Bertrand Lods, Mustapha Mokhtar--Kharroubi},
   pdfsubject={On perturbed substochastic semigroups in  ordered Banach spaces with additive norm.},
  colorlinks=true,
  linkcolor=blue,
  citecolor=blue,
]{hyperref}

 \usepackage{latexsym}
\usepackage{times}
\usepackage{enumerate}
\usepackage{mathrsfs}
\usepackage{stmaryrd}
\usepackage{amsopn}
\usepackage{amsmath}
\usepackage{amssymb}
\usepackage{amsfonts}
\usepackage{amsbsy}
\usepackage{amscd,indentfirst}
\usepackage{hyperref}
\usepackage{amsfonts,amsmath,latexsym,amssymb,verbatim,amsbsy}
\usepackage{amsthm}
\usepackage{colordvi}
\usepackage{mathrsfs}
\usepackage{amsmath,amsthm,amssymb}
\usepackage{latexsym}
\usepackage{times}
\usepackage{enumerate}
\usepackage{mathrsfs}
\usepackage{stmaryrd}
\usepackage{amsopn}
\usepackage{amsmath}
\usepackage{amssymb}
\usepackage{amsfonts}
\usepackage{amsbsy}
\usepackage{amscd,indentfirst}
\usepackage{hyperref}
\usepackage{amsfonts,amsmath,latexsym,amssymb,verbatim,amsbsy}
\usepackage{amsthm}
\usepackage{colordvi}
\usepackage{dsfont}
\renewcommand{\theequation}{\arabic{section}.\arabic{equation}}

\newtheorem{theo}{\sc Theorem}[section]
\newtheorem{lemme}[theo]{\sc Lemma}
\newtheorem{propo}[theo]{\sc Proposition}
\newtheorem{cor}[theo]{\sc Corollary}

\newtheorem{defi}[theo]{\sc Definition}
\newtheorem{proper}[theo]{\sc Properties}
\newtheorem{nb}[theo]{\sc Remark}
\newtheorem{theoB}{\sc Theorem B.}
\newtheorem{lemmeB}[theoB]{\sc Lemma B.}

\newtheorem{defiB}[theoB]{\sc Definition B.}
\newtheorem{exaB}[theoB]{\sc Example B.}
\newtheorem{nbB}[theoB]{\sc Remark B.}

\theoremstyle{definition}

\def \leq {\leqslant}
\def \geq {\geqslant}
\numberwithin{equation}{section}

\def \G {\mathcal{G}}

\def \d {\mathrm{d}}
\def \D {\mathscr{D}}

\def \v {\mathcal{T}}
\def \J {\mathcal{J}}

\def \com {$C_0$-semigroup }

\def \Tt {(\T(t))_{t \geq 0}}
\def \ut {(\u(t))_{t \geq 0}}

\def \vt {(\v(t))_{t \geq 0}}

\def \srt {(\mathcal{S}_r(t))_{t \geq 0}}
\def \st {(S(t))_{t \geq 0}}

\def \ds {\displaystyle}
\def \T {\mathcal{V}}
\def \B {\mathcal{B}}

\def \u {\mathcal{U}}

\def \X {\mathfrak{X}}
\def \la {\left \langle}
\def \ra {\right \rangle}

\def \l {\lambda}
\def \A {\mathcal{A}}
\def \cc {\mathfrak{a}}
\def \bl {\mathbf{\Xi}_\l}
\def \bi {\psi }
\def \tr {\mathrm{tr}}
\def \Tr {\mathrm{tr}}

\def \cS {\mathcal{S}}
\def \cH {\mathcal{H}}

\def \cT {\mathscr{T}}
\def \cts {\mathscr{T}_s(\cH)}

\def \al {algebra }
\def \mA {\mathfrak{A}}

\def \com {$C_0$-semigroup }
\title[]{\textsc{On perturbed substochastic semigroups in abstract state spaces}}

\author{L. Arlotti, B. Lods \& M. Mokhtar-Kharroubi}

\address{{L. Arlotti},  {Dipartimento di Ingegneria Civile e Architettura}, Universit\`a di
Udine, via delle Scienze 208,  33100 Udine, Italy. \newline{\tt
luisa.arlotti@uniud.it}}

\address{{B. Lods}, Laboratoire de Math\'{e}matiques, CNRS UMR 6620, Universit\'{e}
Blaise Pascal (Clermont-Ferrand 2), 63177 Aubi\`{e}re Cedex,
France.\newline {\tt bertrand.lods@math.univ-bpclermont.fr} }
\address{{M. Mokhtar-Kharroubi}, Universit\'e de Franche--Comt\'e, Equipe de Math\'ematiques, CNRS UMR
6623, 16, route de Gray, 25030 Besan\c con Cedex, France.\newline
\noindent{\tt mmokhtar@univ-fcomte.fr}}

\begin{document}

\bibliographystyle{plain}

\maketitle


\begin{abstract}
The object of this paper is twofold: In the first part, we unify and extend
the recent developments on honesty theory of perturbed substochastic
semigroups (on $L^{1}(\mu )$-spaces or noncommutative $L^{1}$ spaces) to
general state spaces; this allows us to capture for instance a honesty
theory in preduals of abstract von Neumann algebras or subspaces of duals of
abstract $C^{\ast }$-algebras. In the second part of the paper, we provide
another honesty theory (a semigroup-perturbation approach) independent of
the previous resolvent-perturbation approach and show the equivalence of the
two approaches. This second viewpoint on honesty is new even in $L^{1}(\mu )$
spaces. Several fine properties of Dyson-Phillips expansions are given and a
classical generation theorem by T. Kato is revisited.


\end{abstract}

\section{Introduction}

In his famous paper on Kolmogorov's differential equations (for Markov
processes with denumerable states) T. Kato  \cite{kato}  introduced the
main tools for dealing with positive unbounded perturbations $\B$ of
generators $\mathcal{A}$ of substochastic semigroups in $\ell^{1}(\mathbb{N})$ provided that a
suitable dissipation  on the positive cone is satisfied. Among other
things, he showed that there exists a unique  extension  $\G\supset
\B+\mathcal{A} $ which generates a substochastic semigroup and characterized the closure
property $\G=\overline{\B+\A}$ by the fact that $\left[ \B(\lambda -\A)^{-1}%
\right] ^{n}\rightarrow 0$ strongly as $n\rightarrow +\infty $ (in general, $%
\G$ may be a proper extension of $\overline{\B+\A}$)$.$ We note
that for ``formally conservative'' equations, such as Kolmogorov's
differential equations, the property \ $\G=\overline{\B+\A}$ is
essential (i.e. necessary and sufficient) to assert that the corresponding
semigroup is \textit{mass-preserving} on the positive cone. Finally, T. Kato
 \cite{kato}  pointed out that his formalism is adapted to general
$AL$--spaces, i.e. Banach lattices $\X$ whose norm is {additive} on
the positive cone $\X_{+}$ i.e. $\left\| x+y\right\| =\left\| x\right\|
+\left\| y\right\| ,\;\;x,y\in \X_{+}.$ Actually, even the lattice assumption
is not essential since Kato's ideas were applied by E. B. Davies  \cite{Da2} to quantum dynamical semigroups in the real Banach space of
self-adjoint trace class operators; in this case, the closure property \ $\G=%
\overline{\B+\A}$ is essential to assert that the corresponding semigroup is
\textit{trace-preserving} on the positive cone.

By the end of the 1980's, Kato's paper \cite{kato} was revisited by
means of Miyadera perturbations in $AL$--spaces \cite{voigt,voigt1, arlotti} and new functional analytic developments followed also in the
2000's \cite{banasiak, BA1, fvdm} which are
known nowadays as the honesty theory of perturbed substochastic semigroups
in $L^{1}(\mu)$ spaces \cite[Chapter 6]{arloban}. Of course, this theory is
motivated by various applications to kinetic theory, fragmentation
equations, birth-and-death equations and so on; see \cite{arloban} and
references therein. We note also that the analysis of piecewise deterministic Markov processes is nicely related to honesty theory in $L^{1}$ spaces
\cite{MTK} (see also \cite{MMTK} for related
topics). On the other hand, in a noncommutative context, there exists also
an important literature (relying on Kato's paper \cite{kato} or some
dual version) on quantum dynamical semigroups, e.g. \cite{Da2,Re, CF, CS, Fa, GQ, GS, S}; such semigroups acting on spaces of
operators arise in the theory of open quantum systems as models of
irreversible (albeit conservative) quantum dynamics. We mention that
quantum dynamical semigroups enjoy the {complete} positivity property
(a stronger property than the fact to leave invariant the positive cone) which
gives their generators a special structure (see e.g. \cite{Fa}).

More recently, in \cite{mkvo}, the honesty theory of perturbed substochastic semigroups in $L^{1}(\mu)$
spaces has been improved and extended in
different directions while a noncommutative version of \cite{mkvo}
was given in \cite{class}. The first goal of the present paper is to provide a
general theory in abstract \textit{state spaces} (i.e. real ordered Banach spaces such that the norm is additive on the positive cone) which covers both \cite{mkvo}
 and \cite{class}. The interest of this abstract approach is not
simply motivated by a unified presentation of \cite{mkvo}
 and \cite{class}: it provides us with an intrinsic treatment of honesty theory in much more
general spaces covering in particular preduals of abstract von Neumann
algebras or more generally subspaces of duals of abstract $C^{\ast }$%
-algebras (see for example \cite{thieme, MK3} on
measure-valued generalization of Kolmogorov equations on abstract measurable
spaces). We refer to E. B. Davies \cite[p. 30-31]{Da1} for the
relevance of the concept of abstract state spaces in probability theory,
quantum statistical mechanics, etc. For its most part, the general theory we
give follows closely \cite{mkvo, class} but we
provide also new informations on the structure of the set of
honest trajectories in the Banach space of bounded measures on a measurable
space and in the Banach space of trace class operators on a Hilbert space.
The second goal of this paper is to provide another approach
of honesty theory. This alternative approach of honesty relies on
Dyson-Phillips expansions (in contrast to the previous resolvent approach)
and is new even in $L^{1}(\mu )$ spaces. To this end, we give several fine
properties of Dyson-Phillips expansions. We also revisit a classical
generation theorem by T. Kato \cite{kato}. Finally, this alternative viewpoint on honesty presents the great advantage of
being adaptable to \textit{nonautonomous} problems \cite{evol}\bigskip

We recall briefly some properties of the class of Banach spaces we shall deal  with in this paper (more information on general real ordered Banach space can be recovered from \cite{depagter, batty}). In all this paper, we shall assume that $\X$ is a real ordered Banach space with a generating positive cone $%
\X_{+}$ (i.e. $\X=\X_{+}-\X_{+}$) on which the norm is
additive, i.e.
$$\left\| u+v\right\| =\left\| u\right\| +\left\| v\right\|\qquad \qquad u,v \in \X_{+}.$$
The additivity of the norm  implies
that the norm is monotone, i.e.
$$0 \leq u \leq v \Longrightarrow \|u\| \leq \|v\|.$$
In particular, the cone $\X_+$ is normal \cite[Proposition
1.2.1]{batty}. It follows easily that any bounded monotone sequence of $\X_+$ is convergent. A property playing an important role in this paper is the existence of a
linear positive functional $\mathbf{\Psi} $ on $\X$ which coincides with the norm on
the positive cone (see e.g. \cite[p. 30]{Da1}), i.e.
\begin{equation}\label{psi}
\mathbf{\Psi} \in \X^\star_+,\qquad \la\mathbf{\Psi}, u \ra
=\|u\|,\qquad u \in \X_+\end{equation}
Note that $\|\mathbf{\Psi}\|= 1.$ Indeed, given $u \in \X$, one has $u=u_1-u_2
\in \X$ with $u_i \in \X_+$ $(i=1,2)$ and $|\la
\mathbf{\Psi},u\ra|=|\,\|u_1\|-\|u_2\|\,| \leq \|u\|$. This proves that
$\|\mathbf{\Psi}\|\leq 1$ and the equality sign follows from \eqref{psi}.
We note also that by a Baire category argument there exists a constant $M >0$ such that each $u \in \X$ has a decomposition $u=u_1-u_2$ where $u_i \in \X_+$ and $\|u_i\|\leq M\|u\|$ $(i=1,2)$; i.e. the positive cone
$\X_+$ is {non-flat}, see \cite[Proposition 19.1]{depagter}. We recall that a \com   $\vt$ of bounded linear operators on $\X$ is called {
substochastic} (resp. stochastic) if $\v(t)$ is positive (i.e. leaves $\X_+$ invariant for any $t\geq 0$)  and $\|\v(t)u\| \leq \|u\|$ (resp. $\|\v(t)u\|=\|u\|$) for all $u \in \X_+$ and $t \geq 0.$ It is not difficult to see that a positive \com $\ut$ with generator $\A$ is substochastic (resp. stochastic) if and only if $\la \mathbf{\Psi},\A u\ra \leq 0$ (resp. $\la \mathbf{\Psi},\A u\ra=0$) for all $u\in \D(\A)_+=\D(\A) \cap \X_+$. Because of a lack ({a priori}) of a lattice structure, $\vt$  need not  be a contraction semigroup. However, one easily sees that $\|\v(t) \| \leq 2M$ for all $t
\geq 0;$ in particular, its type is
nonpositive.\bigskip

The general structure of the paper is the following: our general setting is an abstract state space $\X$, a substochastic \com $\ut$ on $\X$ with generator $\A$ and a linear operator $\B\::\:\D(\A) \to \X$ which is assumed to be positive (i.e. $\B\::\:\D(\A) \cap \X_+ \to \X_+$) and such that
$$\la \mathbf{\Psi}, \A u + \B u \ra \leq 0 \qquad \qquad u \in \D(\A) \cap \X_+.$$
In Section 2, we show that there exists a unique minimal substochastic \com $\Tt$ generated by an extension $\G$ of $\A+\B$. This result was first given by T. Kato \cite{kato} under a  {lattice} assumption on $\X$. Our purpose here is simply to show (by following essentially Kato's ideas) that the lattice assumption is actually unnecessary. We note that this result has been proved differently by means of Miyadera perturbations \cite{thieme} or by using Desch's theorem \cite{mmkl1}. We also show that the corresponding semigroup is given by a (strongly convergent) Dyson-Phillips expansion
$$\T(t)u=\sum_{n=0}^\infty \T_n(t)u$$
without using the theory of Miyadera perturbations. It turns out that the resolvent of $\G$ is given by the strongly convergent series
 $$(\l-\G)^{-1}u=\sum_{n=0}^\infty (\l-\A)^{-1}\left[\B(\l-\A)^{-1}\right]^n u, \qquad \l >0.$$
This series (which does \textit{not} converge a priori in operator norm) is
the corner-stone of a general honesty theory of the \com $\Tt$ given in Section 3 in the spirit of the recent results \cite{class,mkvo}. Besides the functional
$$\cc_0\::\:u \in \D(\G) \to -\la \mathbf{\Psi},\G u\ra$$
and its restriction $\cc$ to $\D(\A)$ we build up and study another functional
$$\overline{\cc}\::\:u \in \D(\G) \to \mathbb{R}$$
which has the properties that $\overline{\cc}_{|\D(\A)}=\cc$ and $\overline{\cc} \leq \cc_0$ on $\D(\G)_+=\D(\G) \cap \X_+$. The trajectory $\left(\T(t)u\right)_{t \geq 0}$ emanating from $u \in \X_+$ is said to be \textit{honest} if
$$\|\T(t)u\|=\|u\|-\overline{\cc}\left(\int_0^t \T(r)u\d r\right), \qquad \qquad \forall t \geq 0$$
or equivalently if
$$\overline{\cc}\left(\int_0^t \T(r)u\d r\right)={\cc_0}\left(\int_0^t \T(r)u\d r\right) \qquad \forall t \geq 0.$$
Various characterization of honesty are given; in particular we show that $\left(\T(t)u\right)_{t \geq 0}$ is honest if and only if $\lim_{n \to \infty}\|\left(\B(\l-\A)^{-1}\right)^nu\|=0$ which is equivalent to $(\l-\G)^{-1}u \in \D(\overline{\A+\B})$. Under the "conservativity" assumption
$$\la \mathbf{\Psi},\A u+\B u\ra=0,\qquad \qquad \forall u \in \D(\A),$$
the mass-preservation in time (i.e. $\|\T(t)u\|=\|u\|$ for any $t \geq 0$) holds if and only if the trajectory $\left(\T(t)u\right)_{t \geq 0}$ is honest. The semigroup $\Tt$ is said to be honest if all trajectories are honest. We show that the honesty of $\Tt$ is equivalent to the identity $\overline{\cc}=\cc_0$ or to the closure property $\G=\overline{\A+\B}.$ Actually, we extend most of the results of \cite{class,mkvo}; in particular we show that the set $\mathcal{H}$ of initial data giving rise to a honest trajectory is a closed hereditary subcone of $\X_+$ and provide a description of the order ideal $\mathcal{H-H}$ (induced by it) in the case where $\X$ is either the Banach space of self-adjoint trace class operators on a Hilbert space or the Banach space of bounded signed measures on a measurable space.

In Section 4, the Dyson-Phillips expansion is the corner-stone of \textit{%
another} honesty theory of trajectories. To this end, we build up and study
a new functional
\[
\widehat{\cc}:u\in \D(\G)\rightarrow
\mathbb{R}
\]%
and show in particular that $\widehat{\cc}_{\mid \D(\A)}=\cc$ and $%
\widehat{\cc}\leq \cc_{0}$ on $\D(\G)_{+}.$ To distinguish a priori the
second notion of honesty from the previous one, we say that a trajectory $(\T(t)u)_{t\geq  0}$ emanating from $u\in \mathcal{X}_{+}$ is
mild honest if%
\[
\left\Vert \T(t)u\right\Vert =\left\Vert u\right\Vert -\widehat{\cc}%
(\int_{0}^{t}\T(r)u\d r),\ \ t\geq 0.
\]%
Various characterizations of mild honesty are given; in particular we show
that $(\T(t)u)_{t\geq 0}$ is mild honest if and only if $%
\int_{0}^{t}\T(r)u\d r\in \D(\overline{\A+\B})$ or
if and only if the integral $ \B\int_{0}^{t}\T_{n}(r)u\d r$ converges strongly to $0$ as $n \to \infty$. This mild honesty is based on
several new fine properties of the operators $\T_{n}$. Finally we
prove that the functionals $\widehat{\cc}$ and $\overline{\cc}$ coincide showing
thus that the notions of honesty and mild honesty are actually equivalent.
Moreover, the equivalence of the two viewpoints on honesty theory provides
us with nontrivial additional results. As we already said it, a honesty
theory in terms of Dyson-Phillips expansions suggests a convenient tool for
the study of nonautonomous problems  \cite{evol}.

\section{Kato's generation theorem and first consequences}\label{sec:kato}

\subsection{Classical Kato's Theorem revisited}

Let $\ut$ be a substochastic \com on $\X$ with generator $\A$. Kato's generation theorem \cite{kato} provides a useful sufficient condition ensuring that some extension of $(\A+\B,\D(\A))$ generates a substochastic \com on $\X$:
\begin{theo}\label{kato}
Let $\ut$ be a substochastic \com on $\X$ with generator $\A$. Let
$\B\::\:\D(\A) \to \X$ be a positive linear operator satisfying:
\begin{equation}\label{hypprinc1}
\la \mathbf{\Psi}, (\A+\B)u \ra \leq 0, \qquad \forall u \in
\D(\A)_+:=\D(\A) \cap \X_+.\end{equation} Then, there exists an
extension $\G$ of $(\A+\B,\D(\A))$ that generates a
substochastic \com $\Tt$ on $X$. Moreover, for any $\lambda > 0$,
the resolvent of $\G$ is given by
\begin{equation}\label{reso}
(\l-\G)^{-1}u=\lim_{n \to \infty} (\l-\A)^{-1}\sum_{k=0}^n
\left[\B(\l-\A)^{-1}\right]^ku , \qquad u \in \X.\end{equation}
Finally, $\Tt$ is the smallest substochastic \com whose generator is an
extension of $(\A+\B,\D(\A))$.
\end{theo}
The general strategy to prove such a result consists in two steps: show that
$$\G_r=\A+r\B,\qquad \D(\G_r)=\D(\A)$$ is a generator of a substochastic \com for any $0 < r < 1$ and \textit{then} use a monotonic convergence theorem by letting $r \nearrow 1$. The first step can be dealt with by means of three different arguments: a direct approach via Hille-Yosida estimates; the use of Miyadera perturbation theory \cite{thieme} or simply the use of Desch theorem \cite{mmkl1}. We revisit here the direct approach via Hille-Yosida estimates by T. Kato \cite{kato}.
 \begin{proof}  Our proof is inspired by the original one of  T. Kato
\cite{kato} that we adapt here to the more general situation we are
dealing with (recall in particular that substochastic semigroups are
contracting \textit{only} on $\X_+$). The proof consists in several
steps.

\noindent $\bullet$ \textit{Construction of $\Tt$:} For any $\l
>0,$ set $\J(\l)=\B(\l-\A)^{-1}$. Clearly, $\J(\l)$ is a bounded
linear positive operator on $\X$ and \eqref{hypprinc1} implies that
\begin{equation*}\begin{split}
\|\J(\l)u\|&=\la \mathbf{\Psi},\J(\l)u\ra \leq -\la   \mathbf{\Psi},\A(\l-\A)^{-1}u\ra \\
 &\leq   \|u\| - \l \|(\l-\A)^{-1}u\| \leq \|u\|, \quad \text{ for any } u \in \X_+ \text{ and any } \l
>0.\end{split}\end{equation*}
Iterating such an inequality leads to
$$\|(\J(\l))^n u\| \leq \|u\|,\qquad \text{ for any } u \in \X_+ \text{ and any } \l
>0, \;n \in \mathbb{N}$$
which implies that \begin{equation*} \|(\J(\l))^n\|
\leq 2M, \qquad \forall n \in \mathbb{N}, \:\l > 0\end{equation*}
where we recall (see the introduction) that $M >0$ is a positive constant such that any $u \in \X$ admits a decomposition $u=u_1-u_2$ with $u_i \in \X_+$ and $\|u_i\|\leq M\|u\|$ $(i=1,2)$. In particular, the spectral radius $r_\sigma(\J(\l))$ of the bounded operator $\J(\l)$ is such that
\begin{equation}\label{rJ}
r_\sigma(\J(\l)) \leq 1, \qquad \forall \l > 0.
\end{equation}
Moreover, the resolvent formula shows that $0 \leq \J(\mu) \leq
\J(\l)$ for any $0 < \l < \mu.$ Now, for any $0 \leq r < 1,$ let us
define  $\G_r$ as
$$\G_r=\A+r\B, \qquad \D(\G_r)=\D(\A).$$  Eq. \eqref{rJ} implies that  $(\l-\G_r)$ is
invertible for any $\l > 0$ with
\begin{equation}\label{kato2.3}
(\l-\G_r)^{-1}=(\l-\A)^{-1}\sum_{n=0}^\infty r^n
\left[\J(\l)\right]^n, \qquad 0 \leq r <1\end{equation} where the
series converges in $\mathscr{B}(\X)$.  For any fixed $f \in \X _+$,
set $v=(\l-\A)^{-1}f$, $\l >0$. One has $v \in \D(\A)_+$ and
\begin{equation*}\begin{split}
\|(\l-\G_r)v\|&=\|(\l-\A-r\B)v\| \geq \|(\l-\A)v\| - r \|\B v\|\\
&=\l \la \mathbf{\Psi}, v \ra - \la \mathbf{\Psi},
\A v \ra - r\la \mathbf{\Psi}, \B v \ra  \geq \l \|v\|.
\end{split}\end{equation*}
Now given $u \in \X_+$ and applying the above reasoning with
$f=\sum_{n=0}^\infty r^n \left[\J(\l)\right]^n u,$  we deduce from
\eqref{kato2.3} that
\begin{equation}\label{substoGr}
\|(\l-\G_r)^{-1}u\| \leq \l^{-1} \|u\|, \qquad \quad \text{ for any
}  u \in \X_+.\end{equation} Iterating this relation, we see that
$$\|\left[(\l-\G_r)^{-1}\right]^n u\| \leq \l^{-n} \|u\|, \qquad
\quad \text{ for any } u \in \X_+ \text{ and any } n \in
\mathbb{N}.$$ Then, since $\X_+$ is non flat, such an estimate
extends to the whole space $\X$ leading to
$$\|\left[(\l-\G_r)^{-1}\right]^n \|\leq \frac{2M}{\lambda^n}, \qquad \forall  \lambda >0,\:\: n \in \mathbb{N},$$ and one deduces from
 {Hille-Yosida Theorem} that, for any $0 \leq r < 1,$
$(\G_r,\D(\A))$ generates a $C_0$-semigroup $\srt$ in $\X$. Since
$(\l-\G_r)^{-1}$ is positive and because of \eqref{substoGr}, $\srt$
is a substochastic \com in $\X$. Moreover, the mapping $r \mapsto
(\l-\G_r)^{-1}u$ is nondecreasing for any fixed $\l >0$ and any $u
\in \X_+$ and one sees from the exponential formula
$$\mathcal{S}_r(t)u=\lim_{n \to \infty} \dfrac{n}{t}\left[\left(\dfrac{n}{t}-\G_r\right)^{-1}\right]^nu, \qquad u \in
\X_+,$$ that the mapping $r \in [0,1) \longmapsto \mathcal{S}_r(t)u$
is also nondecreasing for any fixed $t \geq 0$ and any $u \in \X_+$.
Since $\sup_{0 \leq r < 1}\|\mathcal{S}_r(t)\| \leq 2M$ for any $t
\geq 0$ and any bounded monotone sequence of $\X_+$ is convergent, one gets that
$\mathcal{S}_r(t)$ converges strongly to some operator $\T(t)$ for
any fixed $t \geq 0$ as $r \to 1$. Obviously, $\T(t)$ is a positive
contraction  on $\X_+$ with $\mathcal{S}_r(t) \leq \T(t)$ for any $0
\leq r < 1$ and any $t \geq 0$.

\noindent $\bullet$ \textit{$\Tt$ is a \com on $\X.$} Since
$\mathcal{S}_r(t+s)=\mathcal{S}_r(t)\mathcal{S}_r(s)$ for any $t,s
\geq 0$ and any $0 \leq r <1$, one has, at the limit,
$\T(t+s)=\T(t)\T(s),$ $\forall t,s \geq 0.$ Moreover, $\T(0)=Id$. To
prove that $\Tt$ is a \com on $\X$, it is enough to prove that $t
\geq 0 \mapsto \T(t)u$ is continuous at $t=0$ for any $u \in \X.$
Let us fix $\varepsilon >0$ and $u \in \X_+$. Since $\ut$ is a
strongly continuous, there exists $\delta >0$ such that
$\|\u(t)u-u\| < \varepsilon$ for any $0 \leq t \leq \delta.$ For
such a $t$, we see that, for any $r \in [0,1)$, since
$\mathcal{S}_r(t) \geq \u(t)$, one has
\begin{equation*}\begin{split}
\|\mathcal{S}_r(t)u-\u(t)u\|&=\la  \mathbf{\Psi},
\mathcal{S}_r(t)u-\u(t)u\ra =\la \mathbf{\Psi}, \mathcal{S}_r(t)u\ra
- \la \mathbf{\Psi},\u(t)u\ra\\
&\leq \|u\|-\|\u(t)u\| \leq \|u-\u(t)u\| < \varepsilon.
\end{split}\end{equation*}
One deduces from this estimate that
$$\|\mathcal{S}_r(t)u-u\| \leq \|\mathcal{S}_r(t)u-\u(t)u\| +
\|\u(t)-u\| \leq 2\varepsilon, \qquad \forall 0 < t < \delta.$$
 The important fact is that such
an estimate is uniform with respect to $r \in [0,1)$ so that,
letting $r \nearrow 1,$ one deduces that $\|\T(t)u-u\| \leq
2\varepsilon$ for any $0 < t < \delta.$ This shows that $\lim_{t \to
0}\T(t)u=u$ for any $u \in \X_+$ and, by linearity, the result is
true for any $u \in \X$ which proves that $\T(t)$ is strongly
continuous at $t=0$. 
We denote by $\G$ the generator of $\Tt$. Clearly, $]0,\infty[
\subset \varrho(\G)$ and
$$(\l-\G)^{-1} \text{ is positive}, \quad \|(\l-\G)^{-1}u\| \leq
\|u\|/\l, \qquad u \in \X_+.$$  Note that, since $\mathcal{S}_r(t)
\leq \T(t)$ for any $t \geq 0$ and any $r \in [0,1)$, one also has
$(\l-\G_r)^{-1} \leq (\l-\G)^{-1}$ for any $r \in [0,1)$ and any
$\l
>0$.

\noindent $\bullet$ \textit{$(\l-\G_r)^{-1}$ converges strongly to
$(\l-\G)^{-1}$ as $r \to 1$.} Since for any $u \in \X_+$ the
mapping $r \mapsto \mathcal{S}_r(t)u$ is nondecreasing, by Dini's
Theorem one has for
 any $T
 >0$ and any $u \in \X_+$:
 \begin{equation}\label{dini}\lim_{r \to 1}\sup_{0 \leq t \leq T}
 \|\mathcal{S}_r(t)u-\T(t)u\|=0.\end{equation}
 Now, writing
\begin{multline*}(\l-\G)^{-1}u-(\l-\G_r)^{-1}u=\int_0^T \exp(-\l
t)\left(\T(t)u-\mathcal{S}_r(t)u\right)\d t +\\
\int_T^\infty \exp(-\l t) \left(\T(t)u-\mathcal{S}_r(t)u\right)\d t,
\qquad \forall T \geq 0,\end{multline*} one sees from the uniform
convergence that the first integral converges to $0$ as $r \nearrow
1$ for any $T >0$ while the uniform bound $\sup_{t \geq
0}\|\mathcal{S}_r(t)u-\T(t)u\| \leq 2\|u\|$ allows us to let $T \to
\infty$ in the second integral leading to
$$\lim_{r \to 1}\|(\l-\G_r)^{-1}u-(\l-\G)^{-1}u\|=0, \qquad \forall
\l >0, u \in \X.$$ \noindent $\bullet$ \textit{Proof of Eq.
\eqref{reso}.} Let us fix $\l >0$.  From Eq. \eqref{kato2.3} and the
fact that $0 \leq (\l-\G_r)^{-1} \leq (\l-\G)^{-1}$ for any $0 \leq
r <1$, one has $\mathcal{R}_{r}^{(n)} \leq (\l-\G_r)^{-1} \leq
(\l-\G)^{-1},$ for any $ n \geq 1$  where
$\mathcal{R}_r^{(n)}(\l)=(\l-\A)^{-1}\sum_{k=0}^n r^k
\left[\J(\l)\right]^k.$ Letting $r \nearrow 1$, one gets
$$\mathcal{R}^{(n)}(\l):=(\l-\A)^{-1}\sum_{k=0}^n
\left[\J(\l)\right]^k \leq (\l-\G)^{-1}, \qquad \forall n \geq 1.$$
Since the sequence $\left(\mathcal{R}^{(n)}(\l) \right)_n$ is
nondecreasing, the strong limit $$
\mathcal{R}(\l):=\mathrm{s}-\lim_{n \to
\infty}\mathcal{R}^{(n)}(\l)$$  exists and $\mathcal{R}(\l) \leq
(\l-\G)^{-1}$. We also have $\mathcal{R}^{(n)}_r(\l) \leq
\mathcal{R}^{(n)}(\l) \leq \mathcal{R}(\l)$ for all $0 \leq r < 1$
and $n \geq 1$. Hence, $(\l-\G_r)^{-1}=\mathrm{s}-\lim_{n \to
\infty}\mathcal{R}_r^{(n)}(\l) \leq \mathcal{R}(\l)$ and
$(\l-\G)^{-1}=\mathrm{s}-\lim_{r \to 1}(\l-\G_r)^{-1} \leq
\mathcal{R}(\l).$ This proves finally that
$\mathcal{R}(\l)=(\l-\G)^{-1}$ and Eq. \eqref{reso} is proved.

\noindent $\bullet$ \textit{$\G$ is a closed extension of
$\A+\B$.} With the notation of the previous item, since
$\J(\l)=\B(\l-\A)^{-1}$,  one has
\begin{equation*}\begin{split}\mathcal{R}^{(n)}(\l)&=(\l-\A)^{-1}+(\l-\A)^{-1}\left(\sum_{k=0}^{n-1}[\J(\l)]^k\right)
\B(\l-\A)^{-1}\\
&=(\l-\A)^{-1}+\mathcal{R}^{(n-1)}(\l)\B(\l-\A)^{-1}.
\end{split}\end{equation*}
Thus, for any $u \in \D(\A)$,
$\mathcal{R}^{(n)}(\l)(\l-\A)u=u+\mathcal{R}^{(n-1)}(\l)\B u$ for
any $n \geq 1$. Letting $n \to \infty$, Eq. \eqref{reso} yields
$(\l-\G)^{-1}(\l-\A)u=u+(\l-\G)^{-1}\B u$ or equivalently,
$(\l-\G)^{-1}(\l-\A-\B)u=u.$ In particular, $u \in \D(\G)$ and
$(\l-\G)u=(\l-\A-\B)u$. This proves that $\G$ is an extension
of $\A$ and $\G$ is closed as the generator of a \com on $\X$.

\noindent $\bullet$ \textit{$\Tt$ is minimal.} Let
$(\mathcal{S}(t))_{t \geq 0}$ be a substochastic semigroup  in $\X$
whose generator $\G'$ is a closed extension of $\A+\B$. Let us
prove that $\mathcal{S}(t) \geq \T(t)$ for any $t \geq 0$. Actually,
for any $\l >0$, one has
$$(\l-\G')^{-1}-(\l-\G_r)^{-1}=(\l-\G')^{-1}(\G'-\G_r)(\l-\G_r)^{-1}$$
and, since the range of $(\l-\G_r)^{-1}$ is $\D(\A) \subset \D(\G')
\cap \D(\G_r)$, one
has
\begin{equation*}\begin{split}
(\l-\G')^{-1}-(\l-\G_r)^{-1}&=(\l-\G')^{-1}(\A+\B-\A-r\B)(\l-\G_r)^{-1}\\
&=(1-r)(\l-\G')^{-1}\B(\l-\G_r)^{-1}\end{split}\end{equation*} and
one sees that, at the (strong) limit, $(\l-\G')^{-1} \geq
(\l-\G)^{-1}$. From the exponential formula, one obtains
$\mathcal{S}(t) \geq \T(t)$ for any $t \geq 0$.\end{proof}
\subsection{On Dyson-Phillips expansion series}\label{sec:DP} It is possible to
strengthen the above Theorem \ref{kato} by proving that the
semigroup $\Tt$ is given by a Dyson-Phillips expansion series. Our
approach generalizes the result of \cite{rhandi} to the non lattice
case and relies on different arguments inspired by \cite[Chapter 8]{mmkbook}. We first need some preliminary result. Let us define the
space $\mathscr{C}_{sb}(\mathbb{R}^+, \mathscr{B}(\X))$ of strongly
continuous and bounded mappings $$\cS\::\: t \geq 0 \longmapsto
\cS(t) \in \mathscr{B}(\X)$$  endowed with the norm
$$\|\cS\|_{\infty}=\sup_{t \geq 0}\|\cS(t)\|_{\mathscr{B}(\X)}$$ which
makes it a Banach space. For any $\cS \in
\mathscr{C}_{sb}(\mathbb{R}^+, \mathscr{B}(\X))$, it is possible to
define the time-dependent operator $\mathscr{L}(\cS)(t)$ defined
over $\D(\A)$ by
$$\mathscr{L}(\cS)(t)\::\: u \in \D(\A) \longmapsto \int_0^t \cS(t-s)\B \u(s)u\d s \in
\X,\: t \geq 0.$$ We shall write that $\cS \in \mathscr{C}_{sb}(\mathbb{R}^+, \mathscr{B}^+(\X))$ if $\cS \in \mathscr{C}_{sb}(\mathbb{R}^+, \mathscr{B}(\X))$ and $\cS(t)$ is a positive operator in $\X$ for any $t \geq 0.$ One has the following
\begin{lemme} For any $\cS \in \mathscr{C}_{sb}(\mathbb{R}^+,
\mathscr{B}^+(\X))$ and any $t \geq 0$, $\mathscr{L}(\cS)(t)$ extends
uniquely to a bounded positive operator in $\X$, still denoted
$\mathscr{L}(\cS)(t)$. Moreover, for any $u \in \X$,  the mapping $t
\geq 0 \mapsto \mathscr{L}(\cS)(t)u \in \X$ is continuous.
\end{lemme}
 \begin{proof}  It is clear that $\mathscr{L}(\cS)(t)$
is a nonnegative operator and, for any $u \in \D(\A)_+$ and $\l >0$
one has
$$\left\| \int_0^t \cS(t-s)\B \u(s)u\d s \right\| = \int_0^t
\|\cS(t-s)\B \u(s)u\|\d s  \leq \|\cS\|_\infty \int_0^t
\|\B\u(s)u\|\d s.$$ Now,
\begin{equation}\label{cB}\begin{split}
\int_0^t \|\B\u(s)u\|\d s&=\int_0^t \la \mathbf{\Psi}, \B\u(s)u
\ra
\d s\leq-\int_0^t \la \mathbf{\Psi}, \A\u(s)u \ra \d s\\
&=-\la \mathbf{\Psi}, \int_0^t \A\u(s)u \d s \ra=-\la
\mathbf{\Psi}, \int_0^t \dfrac{\d }{\d s}\u(s)u \d s \ra \\
&= \la \mathbf{\Psi}, u-\u(t)u\ra \leq
\|u\|.\end{split}\end{equation} Therefore,
\begin{equation}\label{cS+}
\left\| \int_0^t \cS(t-s)\B \u(s)u\d s \right\| \leq
\|\cS\|_\infty \|u\| \qquad \forall t \geq 0,\;\forall u \in
\D(\A)_+.\end{equation}
 Now, let $u \in \D(\A)$ be arbitrary and
let $u=u_1-u_2$ where $u_{i}\in \X_+$  are such that $\-u_i\|\leq M\|u\|$, $i=1,2$. Then, for any $n \geq 1$,
$u_n^{i}:=n\int_0^{1/n}\u(s)u_{i}\d s \in \D(\A)_+$ with $u_n^{i}
\longrightarrow u_{i}$ in $\X$ as $n \to \infty$,
 while
$$u_n^1-u_n^2=n\int_0^{1/n}\u(s)u\d s \longrightarrow u \qquad
\text{ in } \D(\A), \qquad i=1,2.$$ Therefore,
\begin{equation*}\begin{split}
\left\| \int_0^t \cS(t-s)\B \u(s)u\d s \right\| &=\lim_{n \to
\infty}\left\| \int_0^t \cS(t-s)\B \u(s)(u_n^1-u_n^2 )\d s \right\|\\
&\leq \lim_{n \to \infty}\left\| \int_0^t \cS(t-s)\B \u(s)u_n^1\d
s \right\| + \\
&\phantom{+++}\lim_{n \to \infty}\left\| \int_0^t \cS(t-s)\B
\u(s)u_n^2\d s \right\|
\end{split}
\end{equation*}
and Eq. \eqref{cS+} yields
\begin{equation*}
\left\| \int_0^t \cS(t-s)\B \u(s)u\d s \right\|  \leq
\|\cS\|_\infty
 \lim_{n \to
\infty}\left(\|u_n^1\|+\|u_n^2\|\right)=\|\cS\|_\infty
\left(\|u_1\|+\|u_2\|\right).\end{equation*} Consequently,
$$\left\| \int_0^t \cS(t-s)\B \u(s)u\d s \right\| \leq 2M
\|\cS\|_\infty  \|u\|, \qquad \forall u \in \D(\A).$$ Since $\D(\A)$ is dense in $\X$,
$\mathscr{L}(\cS)(t)$ extends uniquely to a bounded operator on
$\X$. We still denote $\mathscr{L}(\cS)(t)$ this extension. Notice
that, since $\D(\A)_+$ is dense in $\X_+$, the extension
$\mathscr{L}(\cS)(t)$ is still positive. One notes that, for any $u
\in \D(\A)$, the mapping $t \mapsto \mathscr{L}(\cS)(t)u$ is
continuous. Now, if $u \in \X$, considering a sequence $(u_n)_n
\subset \D(\A)$ which converges to $u$, one has, for any $T >0$
$$\sup_{t \in [0,T]} \left\|\mathscr{L}(\cS)(t)u_n-\mathscr{L}(\cS)(t)u_m\right\| \leq
2M \|\cS\|_\infty  \,\|u_n-u_m\|, \quad n,m \in \mathbb{N},$$ which
implies that the mapping $t \in [0,\infty[ \mapsto
\mathscr{L}(\cS)(t)u$ is continuous.\end{proof}

Arguing as in \cite[Lemma 8.4]{mmkbook}, we prove the following
\begin{theo}\label{rhandi} For any $t \geq 0$, the
following Duhamel formula holds:
\begin{equation}\label{duh}\T(t)u=\u(t)u + \int_0^t \T(t-s)\B\u(s)u \d s, \qquad t \geq 0,
\quad u \in \D(\A).\end{equation} Moreover, the semigroup $\Tt$
defined in Theorem \ref{kato} is given by the
\textit{\textbf{Dyson-Phillips}} expansion series
\begin{equation}\label{dyson}
\T(t)=\sum_{n=0}^{\infty}\mathscr{L}^n(\u)(t), \qquad \qquad t \geq
0\end{equation} where the series converges strongly in
$\X$.\end{theo}
\begin{proof} Let us first establish Duhamel formula.
We use the ideas of \cite[Lemma 1.4]{rhandi}. Let $u \in \D(\A)$ and
$\l > 0.$ We see from \eqref{reso} that
\begin{equation}\label{GBA}(\l-\G)^{-1}u-(\l-\A)^{-1}u=(\l-\G)^{-1}\B(\l-\A)^{-1}u.\end{equation}
Moreover, since $\B$ is $\A$-bounded, the mapping $t \in
[0,\infty) \mapsto \B\u(t)u \in \X$ is continuous for all $u \in
\D(\A)$ and
$$\B(\l-\A)^{-1}u=\B \int_0^\infty \exp(-\l t)\u(t)u \d t= \int_0^\infty
\exp(-\l t) \B \u(t)u \d t.$$ Since $(\l-\G)^{-1}$ is the Laplace
transform of $\Tt$, one gets from \eqref{GBA}
\begin{equation*}\begin{split}
\int_0^\infty \exp(-\l t)\left(\T(t)u-\u(t)u\right)\d
t&=\int_0^\infty\d t \int_0^\infty \exp(-\l(t+s)) \T(t)\B\u(s)u\d s\\
&=\int_0^\infty \exp(-\l t) \left(\int_0^t \T(t-s)\B\u(s)u \d
s\right)\d t.
\end{split}\end{equation*}
Finally, the uniqueness theorem for the Laplace transform provides
the conclusion. Let us prove now that $\Tt$ is given by the
Dyson--Phillips expansion \eqref{dyson}. 
Duhamel formula \eqref{duh} reads
$$\T(t)u=\u(t)u+\mathscr{L}(\v)(t)u, \qquad \forall t \geq 0, \quad
u \in \X$$ and, by iteration,
$$\T(t)u=\sum_{k=0}^n
\mathscr{L}^k(\u)(t)u+\mathscr{L}^{n+1}(\v)(t)u, \qquad t \geq 0,
\quad n \geq 1, \quad u \in \X.$$ In particular, for any $u \in
\X_+$, one has
\begin{equation}\label{sommen}\sum_{k=0}^n
\mathscr{L}^k(\u)(t)u \leq \T(t)u, \qquad n \geq 1, \quad u \in
\X_+\end{equation} and the series $\sum_{n=0}^\infty \mathscr{L}^n (\u)u $ is
convergent towards a limit that we denote $\widetilde{\v}(t)u.$
Notice that, for a given $u \in \X_+$, the mapping $t \in [0,\infty[
\mapsto \widetilde{\v}(t)u$ is measurable. One has
\begin{equation}\label{tilde}
\widetilde{\v}(t)u \leq \T(t)u, \qquad \forall u \in \X_+, \:t \geq
0.\end{equation} Now, it is not difficult to check by induction that
\begin{equation}\label{Ln}\int_0^\infty \exp(-\lambda t) \mathscr{L}^n(\u)(t) u \d t =
(\l-\A)^{-1}\left[\B(\l-\A)^{-1}\right]^nu\end{equation} so that,
$$\sum_{n=0}^\infty
(\l-\A)^{-1}\left[\B(\l-\A)^{-1}\right]^nu=\int_0^\infty
\exp(-\lambda t)\widetilde{\v}(t)u \d t$$ and Eq. \eqref{tilde}
together with Eq. \eqref{reso} yield
$$\int_0^\infty \exp(-\lambda t)\widetilde{\v}(t)u \d
t=\int_0^\infty \exp(-\lambda t)\T(t) u \d t, \qquad \forall u \in
\X_+, \quad \l >0.$$ The uniqueness theorem for the Laplace
transform implies then $\widetilde{\v}(t)u=\T(t)u$ for any $t \geq
0$ and any $u \in \X_+$ so that
$$\sum_{n=0}^\infty \mathscr{L}^n(\u)(t)u=\T(t)u, \qquad \forall u \in \X_+,
\quad t \geq 0.$$ Note that, according to Dini's convergence
theorem, the series converges uniformly in bounded time. One extends
then the convergence to arbitrary $u \in \X$ by linearity.
\end{proof}
\begin{nb}\label{remark24}
Notice that the family of operators $\T_n(t)=\mathscr{L}^n(\u)(t)  $
 $(n \in \mathbb{N},$ $t \geq 0)$,
 is nothing but the classical
Dyson-Phillips iterated usually defined by induction \cite[Chapter
7]{mmkbook}:
\begin{equation}\label{dysonphi2}\T_{n+1}(t)u=\int_0^t \T_{n } (t-s)\B\u(s)u\d s, \qquad \forall n \in \mathbb{N}, \:u \in \D(\A).\end{equation}
Notice that, according to \eqref{sommen}, one sees easily that
\begin{equation}\label{luisanotes1}
\sum_{k=0}^n \|\T_k(t)u\| \leq \|u\| \quad \text{ for any } \quad t \geq 0, \:\:u
\in \X_+.
\end{equation}
Moreover, for any $n \in \mathbb{N},$ the mapping $t \in [0,\infty) \mapsto \T_n(t)u$ is continuous for any $u \in
\X$. Finally, arguing as in \cite[p. 129]{arloban}, it is not difficult to prove that, for any $n \in \mathbb{N}$, the following relation holds:
\begin{equation}\label{luisanotes2}
\T_n(t+s)u=\sum_{k=0}^n\T_k(t)\T_{n-k}(s)u  \quad \text{ for any } u \in
\X, \:\:  t,s \geq 0.
\end{equation}
\end{nb}

\section{On honesty theory: resolvent approach}
From now, in all the paper, we assume that the assumptions of Theorem
\ref{kato} are  met.

\subsection{About some useful  functionals}

 Since the \com $\Tt$ is substochastic, one has, for any $u \in
 \X_+$,
 $$ \la \mathbf{\Psi}, \T(t)u-u\ra=\|\T(t)u\| - \|u\| \leq 0, \qquad
 \forall t \geq 0,\quad u \in \X_+.$$ In particular, if one chooses
 $u \in \D(\G)_+$ here above, since,
$$\la \mathbf{\Psi},\G  u \ra=\lim_{t \searrow 0}t^{-1}\la
\mathbf{\Psi}, \T(t)u-u\ra$$ one gets
\begin{equation}\label{eq:5}
\la \mathbf{\Psi}, \G u \ra \leq 0, \qquad \qquad u \in \D(\G)_+.
\end{equation}
 Because of this elementary but fundamental inequality, a
crucial role in the present approach will be based on the properties
of the following functional:
 $$\cc_0\:\::\:\:u \in \D(\G) \mapsto \cc_0(u)=-\la \mathbf{\Psi},\G
u \ra \in \mathbb{R}.$$ Because of \eqref{eq:5}, this functional
$\cc_0$ is nondecreasing, i.e. $\cc_0(u) \geq \cc_0(v)$ for any $
u,v\in \D(\G)$ with $u\geq v.$ Moreover, since $\|\mathbf{\Psi}\|\leq 1$, one has $\cc_0(u) \leq \|\G u\|$ for any $u \in \D(\G)$. We denote by $\cc$ its restriction
to $\D(\A)$, i.e.
$$\cc\:\::\:\:u \in \D(\A) \mapsto \cc (u)=-\la \mathbf{\Psi},\A u + \B
u \ra \in \mathbb{R}.$$ Let $\l >0$ be \textit{fixed}. The following
obvious identity
\begin{equation}\label{eq:6}
-\cc((\l-\A)^{-1}u)=\l
\|(\l-\A)^{-1}u\|+\|\B(\l-\A)^{-1}u\|-\|u\|,
\end{equation}
is valid for any $u \in \X_+$. Moreover,  the sequence
$\left(\sum_{k=0}^n (\l-\A)^{-1}[\B(\l-\A)^{-1}]^k
u\right)_n$  is nondecreasing and convergent to $(\l-\G)^{-1}u$.
Since $\cc(\cdot)$ is nondecreasing, one gets
$$\cc\left(\sum_{k=0}^n (\l-\A)^{-1}[\B(\l-\A)^{-1}]^k u\right) \leq
\cc_0((\l-\G)^{-1}u),$$ for  all $u \in \X_+$ and any $n \in
\mathbb{N}.$ The bounded and nondecreasing real sequence
$$\left(\cc\left(\sum_{k=0}^n (\l-\A)^{-1}[\B(\l-\A)^{-1}]^k
u\right)\right)_n$$ is therefore convergent. This convergence holds
for any $u \in \X=\X_+-\X_+$ and therefore defines a functional
$\overline{\cc}_\l$ (that depends {\it a priori} on $\l
>0$) on the domain of $\G$ by
$$\overline{\cc}_\l\left((\l-\G)^{-1}u\right) =\sum_{n=0}^\infty
\cc\left((\l-\A)^{-1}\left[\B(\l-\A)^{-1}\right]^nu\right),
\qquad u \in \X.$$ Following \cite{mkvo}, we derive another
expression for $\overline{\cc}_\l$ from the identity
$$(\l-\G_r)^{-1}u=\sum_{n=0}^\infty r^n
(\l-\A)^{-1}\left[\B(\l-\A)^{-1}\right]^nu, \qquad u \in
\X_+$$ established in the proof of Theorem \ref{kato}. We recall
that, denoting $\D_\A$ and $\D_\G$ the domain of $\A$ and $\G$
equipped with their respective graph norm,  the series is convergent
in $\D_\A$ and, since $(\l-\A)^{-1} \leq (\l-\G)^{-1}$, the
embedding $\D_\A \hookrightarrow \D_\G$ is continuous. Therefore,
$$\cc((\l-\G_r)^{-1}u)=\sum_{n=0}^\infty r^n
\cc\left((\l-\A)^{-1}\left[\B(\l-\A)^{-1}\right]^nu\right),
\qquad u \in \X_+.$$ Letting now $r \to 1$, one gets
$$\overline{\cc}_\l\left((\l-\G)^{-1}u\right)=\lim_{r \nearrow 1}\cc((\l-\G_r)^{-1}u)=\sum_{n=0}^\infty
\cc\left((\l-\A)^{-1}\left[\B(\l-\A)^{-1}\right]^nu\right).$$
One has the following basic result which can be proved exactly
as \cite[Prop. 1.1]{mkvo} (see also an alternative proof at the end of the paper, Theorem \ref{final}):
\begin{propo}\label{prop1.1.mkvo} Let $0 < \l < \mu$. Then,
\begin{enumerate}
\item ${\overline{\cc}_\l}_{|\D(\A)}=\cc;$
\item $\overline{\cc}_\l=\overline{\cc}_\mu .$
\end{enumerate}
This  defines a  functional $\overline{\cc} :=\overline{\cc}_\l$ for
any $\l.$ 
\end{propo}
\begin{nb} Let us point out that $\overline{\cc}$ is continuous with respect to
the graph norm of $\G$.
\end{nb}
The above definitions of functionals $\overline{\cc}$ and $\cc_0$
lead to the following:
\begin{defi}\label{defi:bl} For any $\l >0$, we define the functional $\bl \in
\X ^\star$ by
$$\la \bl, u \ra =\cc_0
\left((\l-\G)^{-1}u\right)-\overline{\cc}\left((\l-\G)^{-1}u\right),
\qquad u \in \X.$$
\end{defi}
One has the following  Lemma:
\begin{lemme}\label{lemm1.4.mkvo} For any $\l >0$ and $u \in \X$
\begin{equation*}
\la \bl, u \ra=\lim_{n \to \infty} \la \mathbf{\Psi},
\left[\B(\l-\A)^{-1}\right]^n u \ra=\lim_{r \nearrow 1}(1 - r)\la
\mathbf{\Psi}, \B(\l-\G_r)^{-1}u \ra
\end{equation*}
\end{lemme}
{\begin{proof} One has to compute $\la \bl,
u\ra=\cc_0\left((\l-\G)^{-1}u\right)-\overline{\cc}\left((\l-\G)^{-1}u\right)$.
First,
\begin{equation*}\begin{split}
\overline{\cc}\left((\l-\G)^{-1}u\right)&=\sum_{n=0}^\infty
\cc\left((\l-\A)^{-1}\left(\B(\l-\A)^{-1}\right)^n u\right)\\
&=\sum_{n=0}^\infty \la \mathbf{\Psi},
-(\A+\B)(\l-\A)^{-1}\left(\B(\l-\A)^{-1}\right)^n
u\ra.
\end{split}
\end{equation*}
Now, the latter is equal to
$$\sum_{n=0}^\infty \bigg(\la \mathbf{\Psi},
\left(\B(\l-\A)^{-1}\right)^n u
-\left(\B(\l-\A)^{-1}\right)^{n+1}u-\l
(\l-\A)^{-1}\left(\B(\l-\A)^{-1}\right)^n u\ra\bigg).$$
Thus
\begin{equation*}\begin{split}
\overline{\cc}\left((\l-\G)^{-1}u\right)&=\la \mathbf{\Psi}, u \ra -
\lim_{n \to \infty} \la
\mathbf{\Psi},\left(\B(\l-\A)^{-1}\right)^n
u\ra - \\
&\phantom{++++ +++}\l \la \mathbf{\Psi}, \sum_{n=0}^\infty (\l-\A)^{-1}\left(\B(\l-\A)^{-1}\right)^n u\ra\\
&=\la \mathbf{\Psi}, u \ra - \lim_{n \to \infty} \la
\mathbf{\Psi},\left(\B(\l-\A)^{-1}\right)^n u\ra - \l \la
\mathbf{\Psi},
(\l-\G)^{-1}u\ra\\
&=\cc_0\left((\l-\G)^{-1}u\right)-\lim_{n \to \infty} \la
\mathbf{\Psi},\left(\B(\l-\A)^{-1}\right)^n u\ra
\end{split}
\end{equation*}
which proves the first assertion. On the other hand,
\begin{equation*}\begin{split}
\overline{\cc}\left((\l-\G)^{-1}u\right)&=\lim_{r \nearrow 1}
\cc\left((\l-\G_r)^{-1}u\right)\\
&=\lim_{r \nearrow 1} \la \mathbf{\Psi}, (\l-\A-r\B-\l-(1-r)\B)(\l-\G_r)^{-1}u\ra\\
&=\lim_{r \nearrow 1} \bigg(\la \mathbf{\Psi}, u \ra-\l \la
\mathbf{\Psi}, (\l-\G_r)^{-1}u\ra -(1-r)\la \mathbf{\Psi},
\B(\l-\G_r)^{-1}u\ra\bigg)\\
&=\la \mathbf{\Psi}, u \ra - \l \la \mathbf{\Psi}, (\l-\G)^{-1}u\ra
-\lim_{r \nearrow 1} (1-r)\la \mathbf{\Psi},
\B(\l-\G_r)^{-1}u\ra
\end{split}\end{equation*}
provides the second assertion.
\end{proof}}

We end this section with the following fundamental result:
\begin{theo}\label{equivalence} Let $\l >0$ and $u \in \X_+$ be fixed. The following assertions
  are equivalent:
\begin{enumerate}[(i)\:]
\item the set $\left\{[\B(\l-\A)^{-1}]^n u \right\}_n$ is relatively weakly compact;
\item $\lim_{n \to \infty} \|[\B(\l-\A)^{-1}]^n u\|=0$;
\item $\la \bl, u \ra=0$;
\item $(\l-\G)^{-1}u \in \D(\overline{\A+\B}).$
\end{enumerate}
\end{theo}
\begin{proof}It is clear from the definition of $\bl$ that $\textit{(ii)} \implies
 \textit{(iii)}$  and that $\textit{(iii)} \implies \textit{(ii)} \implies
 \textit{(i)}.$

\textit{(i)} $\implies$ \textit{(ii)} and \textit{(iv)}.   Let
$v_n:=\sum_{k=0}^n(\l-\A)^{-1}\left[\B(\l-\A)^{-1}\right]^k
u.$ Clearly, $v_n \in \D(\A+\B)$ and $v_n$ converges to
$v=(\l-\G)^{-1}u$ in $\X$ as $n$ goes to infinity. Moreover, it is
not difficult to see that
$$(\l-\A-\B)v_n=u-\left[\B(\l-\A)^{-1}\right]^{n+1}u.$$
If some subsequence
$\left([\B(\l-\A)^{-1}]^{n_k}u\right)_k$ converges weakly
in $\X$ to some $z \in \X$, then $ (\l-\A-\B)v_{n_k}$
converges weakly to $u-\B(\l-\A)^{-1}z$ as $k \to \infty$.
It follows from the weak closedness of the graph
$\overline{\A+\B}$ that $v \in
\D(\overline{\A+\B})$ and
$$(\l-\overline{\A+\B})v=u-\B(\l-\A)^{-1}z.$$
Since $\G$ is a closed extension of $\A+\B$ and
$v=(\l-\G)^{-1}u$, the latter reads
$$u=u-\B(\l-\A)^{-1}z$$
so that $\B(\l-\A)^{-1}z=0$. Hence,
$[\B(\l-\A)^{-1}]^{n_k+1}u$ converges weakly to $0$ as $k
\to \infty$. In particular,
$$\lim_{k \to \infty}\la \mathbf{\Psi},\bigg[ \B(\l-\A)^{-1}\bigg]^{n_k+1}u\ra =0 \quad \text { and
}\quad \lim_{n \to \infty}\la \mathbf{\Psi},\bigg[
\B(\l-\A)^{-1}\bigg]^{n}u\ra=0$$ since the whole sequence
is always convergent. This proves \textit{(ii)}.  Notice also that
$v =(\l-\G)^{-1} u \in
  \D(\overline{\A+\B})$ and \textit{(iv)}  is proved.

\textit{(iv)} $\implies$ \textit{(iii)}. One can assume without loss
of generality that $\bl \neq 0$. Assume that $(\l-\G)^{-1} u \in
\D(\overline{\A+\B})$. According to the following identity
(see \cite[Lemma 4.5, p. 117]{arloban})
\begin{equation}\label{jacek}
\D(\overline{\A+\B})=(\l-\G)^{-1}\overline{(I-\B(\l-\A)^{-1})\X}
\end{equation}
one sees that there exists a sequence $(u_n)_n \subset
(I-\B(\l-\A)^{-1})\X$ such that $\lim_n u_n =u$. It is
easy to see that $\la \bl, u_n \ra=0$ for any $n \in \mathbb{N}$ so
that $\la \bl,u \ra=0.$\end{proof} One deduces from the above result
that $\D(\A+\B)$ is a core for $\G$ if and only if
$\bl=0$:
\begin{cor}\label{coro:equiva} One has $\G=\overline{\A+\B}$ if and only if $\bl=0$ for some (or
  equivalently for all) $\l >0.$\end{cor}

\begin{nb} For $v \in \D(\G)_+$ one can show as in \cite[Proposition
  1.6]{mkvo} that $v \in \D(\overline{\A+\B})$ if and only if
  $\cc_0(v)=\overline{\cc}(v)$ which strengthens Proposition \ref{prop1.1.mkvo}.\end{nb}

\subsection{On honest trajectories}
We note that, for any $u \in \X_+$ and any $t  \geq 0$, one has
$$
\ds\int_0^t \T(s)u \d s \in \D(\G)\quad  \text{ with }\quad \T(t)u-
u=\G\int_0^t \T(s)u \d s.$$ Since the semigroup is positive, one has
\begin{equation}\label{equalityG}\|\T(t)u\|-\| u\|=  -\cc_0\left(\int_0^t \T(s)u \d
s\right).\end{equation}
\begin{defi}\label{defi:hon}
Let $u \in \X_+$ be given. Then, the trajectory
$(\T(t)u)_{t \geq 0}$ is said to be \textit{\textbf{honest}} if and
only if
$$\|\T(t)u\|=\| u\|-\overline{\cc}\bigg(\int_0^t \T(s)u \d
s\bigg),\quad \text{ for any }   t\geq 0.$$  The whole
$C_0$-semigroup $\Tt$ will be said to be honest if all trajectories
are honest.
\end{defi}

\begin{nb} Note that, in the spirit of \cite{mkvo}, it is
possible to define a more general concept of \textit{local} honest trajectory on an
 interval $\mathcal{I} \subset [0,\infty)$ by
$$\overline{\cc}\bigg(\ds \int_s^t \T(r)u\d r\bigg)=\cc_0\bigg(\ds
\int_s^t \T(r)u \d r\bigg), \qquad \text{ for any } t,s \in
\mathcal{I},\:t \geq s.$$ We do not try to elaborate on this point here.
\end{nb}

\begin{nb}\label{nb:cci} One can deduce from Theorem \ref{equivalence} and Corollary \ref{coro:equiva}
the following: given $u \in \X_+$, one sees from  \eqref{equalityG}
that $(\T(t)u)_{t \geq 0}$ is honest if and only if
$$\overline{\cc}\bigg(\ds \int_s^t \T(r)u\d r\bigg)=\cc_0\bigg(\ds
\int_s^t \T(r)u \d r\bigg) \text{ for any  } t \geq s \geq 0.$$ Moreover,
it is easy to see that this is equivalent to $\overline{\cc} (
\int_0^t \T(r)u\d r )=\cc_0 ( \int_0^t \T(r)u \d r )$ for any $t
\geq 0.$
\end{nb}

The link between honest trajectory and the functional $\bl$ given by
Definition \ref{defi:bl} is provided by the following:
\begin{theo}\label{char}
Let $u \in \X_+$. The trajectory $(\T(t)u)_{t \geq 0}$ is honest if
and only if $\la \bl, u\ra =0$ for all/some $\l>0.$
\end{theo}
 {\begin{proof} We recall that, for any $\l >0$,
\begin{equation}\label{integral}(\l-\G)^{-1} u = \int_0^\infty \exp(-\l t)\T(t)u \d t=\l
\int_0^\infty \exp(-\l t) \bigg(\int_0^t \T(s)u \d s\bigg)\d
t.\end{equation} Moreover, the function $t \longmapsto\int_0^t
\T(s)u d s$ is continuous and linearly bounded as a
$\D_\G$-function. This means that the above outer integral in
\eqref{integral} is convergent in $\D_\G$ and commute with $\cc_0$.
Moreover, according to Prop. \ref{prop1.1.mkvo}, it also commutes
with $\overline{\cc}$ so that
$$\cc_0\bigg((\l-\G)^{-1}u\bigg)=\l \int_0^\infty \exp(-\l t)
\cc_0   \bigg(\int_0^t \T(s)u \d s\bigg) \d t$$ and
$$\overline{\cc}\bigg((\l-\G)^{-1}u\bigg)=\l \int_0^\infty \exp(-\l t)
\overline{\cc}  \bigg(\int_0^t \T(s)u \d s\bigg) \d t.$$ One sees
therefore that $\cc_0   \bigg(\ds \int_0^t \T(s)u \d
s\bigg)=\overline{\cc} \bigg(\ds \int_0^t \T(s)u \d s\bigg)$ for any
$t \geq 0$ is equivalent to
$\cc_0\bigg((\l-\G)^{-1}u\bigg)=\overline{\cc}\bigg((\l-\G)^{-1}u\bigg)$
for any $\l >0$ and proves the Theorem.
\end{proof}}

\begin{nb} Notice that the whole semigroup $\Tt$ is honest if and only
$\G=\overline{\A+\B}$ and this is also equivalent to
$\bl=0$ for some / all $\l >0$.
\end{nb}

\subsection{On an order ideal invariant under $\Tt$}
We already know that, for any $u \in \X_+$, the property $\la \bl,
u\ra
 =0$ is independent of the choice of $\l >0.$ This allows us to
 define the set
 \begin{equation}\label{H1}
 \mathcal{H}=\bigg\{u \in \X_+\,;\,\la \bl,u\ra=0 \text{
for any } \l >0 \bigg\}.\end{equation} Notice that, by virtue of
Theorem \ref{char}, $\mathcal{H}$ is precisely the set of initial positive
data $u$ giving rise to honest trajectories:
$$\mathcal{H}=\bigg\{u \in \X_+\, ;\,(\T(t)u)_{t \geq
0} \text{ is honest}\bigg\}.$$
 One has
the following
\begin{propo}\label{closed} The set $\mathcal{H}$ is  invariant
under $\Tt$ and $(\l-\G)^{-1}$ $(\l >0)$.  Moreover, for any $u \in
\mathcal{H}$, if $\,\mathcal{I}_u=\left\{z \in \X_+\,;\,\exists p
\in \mathbb{R}_+ \text{ such that } \, pu-z \in \X_+\right\}$ then
$\overline{\mathrm{span}(\mathcal{I}_u)} \cap \X_+ \subset
\mathcal{H}$.
\end{propo}

 \begin{proof} Let $u \in \mathcal{H}$. This means that
$$\|\T(t)u\|-\|u\|=-\overline{\cc}\bigg(\int_0^t \T(s)u \d s\bigg), \qquad \forall t \geq 0.$$
Let $t_0 >0$ be fixed and set $v=\T(t_0)u$. One has
$\|v\|-\|u\|=-\overline{\cc}\bigg(\ds\int_0^{t_0} \T(s)u \d s\bigg)$
and, for any $t \geq t_0$,
$$\|\T(t-t_0)v\|-\|u\|=-\overline{\cc}\bigg(\int_0^t \T(s)u \d
s\bigg)=-\overline{\cc}\bigg(\int_0^{t_0} \T(s)u \d
s\bigg)-\overline{\cc}\bigg(\int_{t_0}^t \T(s)u \d s\bigg)$$ so that
$$\|\T(t-t_0)v\|=\|v\|-\overline{\cc}\bigg(\int_{t_0}^t \T(s)u \d
s\bigg)=\|v\|-\overline{\cc}\bigg(\int_{0}^{t-t_0} \T(s)v \d
s\bigg), \qquad \forall  t \geq t_0.$$ In other words, $v \in
\mathcal{H}$ and $\mathcal{H}$ is invariant under the action of
$\Tt$. Let $\l >0$ and   $u \in \mathcal{H}$ be fixed. One has
$\cc_0 ((\l-\G)^{-1}u )=\overline{\cc} ((\l-\G)^{-1}u )$ and $\cc_0
((\mu-\G)^{-1}u )=\overline{\cc} ((\mu-\G)^{-1}u )$ for any $\mu
>0$. One sees as a direct application of the resolvent formula  that
$$\cc_0\bigg((\mu-\G)^{-1}(\l-\G)^{-1}u\bigg)=\overline{\cc}\bigg((\mu-\G)^{-1}(\l-\G)^{-1}u
\bigg), \quad \forall \mu >0$$ which amounts to $(\l-\G)^{-1}u \in
\mathcal{H}.$ Finally,   let $u \in \mathcal{H}$ and $z \in
\mathcal{I}_u$ be fixed, there is some nonnegative real number $p$
such that $p u-z \in \X_+$. Then, for any $n \in \mathbb{N}$,
$$[\B(\l-\A)^{-1}]^{n+1}z \leq p [\B(\l-\A)^{-1}]^{n+1}u.$$
Since  $\la \mathbf{\Xi}_\l,u \ra=0$, Lemma \ref{lemm1.4.mkvo}
clearly implies that $$\lim_{n \to \infty} \la \mathbf{\Psi},
\left[\B(\l-\A)^{-1}\right]^{n+1}z \ra =0$$  and
$(\T(t)z)_{t \geq 0}$ is honest according to Theorem
\ref{equivalence}. This proves that $\mathcal{I}_u \subset
\mathcal{H}$ and, since $\bl$ is a continuous and positive linear
form on $\X$, one deduces easily that
$\overline{\mathrm{span}(\mathcal{I}_u)} \cap \X_+ \subset
\mathcal{H}$.
\end{proof}

Thanks to the above structure of $\mathcal{H}$, it is possible to
provide sufficient conditions ensuring that the whole semigroup is honest.
\begin{theo}\label{theo:irr}\mbox{}
\begin{enumerate}
\item If $\mathcal{H}$ contains a
quasi-interior element $u$, then the whole semigroup $\Tt$ is
honest.
\item Assume $\Tt$ to be irreducible. Let there exists $u \in
\X_+ \setminus \{0\}$ such that $(\T(t)u)_{t \geq 0}$ is honest.
Then, the whole semigroup $\Tt$ is honest.\end{enumerate}
\end{theo}
\begin{proof} \textit{(1)} If $\X$ contains a quasi-interior element $u$, then
\cite{schaef, depagter} $\overline{\mathrm{span}(\mathcal{I}_u)}=\X_+
$. One sees then that, if $u \in \mathcal{H}$, Proposition \ref{closed} implies $\mathcal{H}=\X_+$.

\textit{(2)} According to Proposition \ref{closed}, $\mathcal{H}$ is
invariant by $(\l-\G)^{-1}$ for any $\l >0.$ Therefore,
$v=(\l-\G)^{-1}u$ is a quasi-interior element of $\mathcal{H}$ and
we conclude by the first point.\end{proof}
Before giving some more precise properties of $\mathcal{H}$ let us introduce the notions of ideal and hereditary subcone:
\begin{defi} A subcone $\mathcal{C}$ of $\X_+$ is said to be \textit{\textbf{hereditary}} if $0 \leq u \leq v$ and $v \in \mathcal{C}$ imply $u \in \mathcal{C}.$ An \textbf{order ideal} of $\X$ is a linear subspace $\mathscr{A}$ of $\X$ such that $u_1 \leq v \leq u_2$ and $u_i \in \mathscr{A}$, $i=1,2$ imply $v \in \mathscr{A}.$ An order ideal $\mathscr{A}$ of $\X$ is said to be\textbf{ positively generated} if $\mathscr{A}=\left(\mathscr{A} \cap \X_+\right)-\left(\mathscr{A} \cap \X_+\right).$
\end{defi}
\begin{nb}\label{remarkhere} Notice that, if $\mathscr{A}$ is a positively generated order ideal of $\X$ then
$$u \in \mathscr{A} \quad \implies \quad |u| \in \mathscr{A}.$$
Indeed, since $\mathscr{A}$ is positively generated one has $u=u_1-u_2$ with $u_i \in \mathscr{A} \cap \X_+$. Moreover, according to \cite[Lemma 2]{robinson}, $\mathscr{A} \cap \X_+$ is an hereditary subcone of $\X_+$. In particular, since $0 \leq |u| \leq u_1+u_2$ one gets $|u| \in \mathscr{A} \cap \X_+.$
\end{nb}
The subset $$\mathscr{H}:=\mathcal{H}-\mathcal{H}$$ enjoys the following properties:
\begin{theo}\label{structure} Let $\mathcal{H}$ be defined by \eqref{H1}. Then,
$\mathcal{H}$ is a closed hereditary subcone of $\X_+$ and
$\mathscr{H}$ is an order ideal with induced
positive cone $\mathscr{H}_+$ equal to $\mathcal{H}$. Moreover,
$\mathscr{H}$ is invariant under $\Tt$.
\end{theo}

\begin{proof}
We first note that, since $\bl$ is a positive and continuous linear
form over $\X$,
$$\mathcal{H}=\bigg\{u \in \X\,;\,\la \bl, u \ra =0  \text{
for any } \l >0\bigg\} \cap \X_+$$ is clearly a closed convex
subcone of $\X_+$. Moreover, if $0 \leq u \leq v$ with $v \in
\mathcal{H}$ then, for any $\l >0$, $\la \bl, v \ra =0$ and
consequently $\la \bl, u \ra =0$ since $\bl$ is positive, i.e.
$\mathcal{H}$ is a closed hereditary subcone of $\X_+$. It is easy
to see that $\mathscr{H}:=\mathcal{H}-\mathcal{H}$ is the linear
space generated by $\mathcal{H}$. Then, by \cite[Lemma 2]{robinson},
$\mathscr{H}$ is an order ideal with positive cone $\mathcal{H}$.
The fact that $\mathscr{H}$ is invariant under the semigroup $\Tt$
follows from the previous Proposition.\end{proof}

A priori, in the general setting above, it is \textit{not clear}
that $\mathscr{H}$ is closed in $\X$. However, we have more precise
results in $AL$-spaces (i.e. Banach lattices with additive norm) and
in preduals of von Neumann algebras.
\begin{propo}\label{propo:von} (i) If $\X$ is a  $ {AL} $-space then $\mathscr{H}$ is a closed
lattice ideal (and therefore a projection band) of $\X$. In particular, there exists  a band projection $\mathbf{P}$
onto $\mathscr{H}$   such that $\mathscr{H}=\mathbf{P}\X$ and  $\X=\mathscr{H}
\oplus \mathscr{H}_d$ where the
disjoint complement $\mathscr{H}_d$  of $\mathscr{H}$  is given by
$\mathscr{H}_d=(I-\mathbf{P})\X$.

(ii) Let $\X$ be the predual of a von Neumann algebra. Then,
$\mathscr{H}$ is a closed order ideal.\end{propo}
\begin{proof}
(i) Let $(u_n)_n \subset \mathscr{H}$ be such that $u_n \to u$ in
$\X$. By assumption, $u_n=v_n-w_n$ with $v_n,\,w_n \in \mathcal{H}$.
In particular, $|u_n| \leq v_n + w_n $ and  $\la \bl,|u_n| \ra \leq
\la \bl, v_n\ra + \la \bl, w_n \ra=0$  whence $|u_n| \in
\mathcal{H}$. It follows that the negative and positive parts
$u_n^-$ and $u_n^+$ both belong to $\mathcal{H}$. Since $\X$ is a
vector lattice, the mappings $v \in \X \mapsto v^{\pm} \in \X_+$ are
continuous \cite[Proposition 5.2]{schaef}, one has $u_n^{\pm}  \to u^{\pm}$
and $u^+,u^-$ belong to $\mathcal{H}$. This proves that $u=u^+-u^-
\in \mathscr{H}.$

(ii) If $\mathfrak{A}$ is a von Neumann algebra and
$\X=\mathfrak{A}_\star$ is its predual, then the mapping $u \in \X
\mapsto |u| \in \X_+$ is continuous (see e.g. \cite[Proposition
4.10, p. 415]{takesaki} ) and then, arguing as in (i), one gets the
conclusion.
\end{proof}
\begin{nb}\label{hdd} In the above case (i), the positive cone of the disjoint complement $\mathscr{H}_d$
does not contain non-trivial elements with a honest trajectory. In particular, dishonest trajectories are all emanating from elements of the positive cone of  $\X=\mathscr{H} \oplus \mathscr{H}_d$ having a non-trivial component over $\mathscr{H}_d.$
\end{nb}
We now deal with two practical examples for concrete spaces:
\subsubsection*{Example 1: The space of bounded signed measures} Let $\left(\Sigma, \mathcal{F}\right)$ be a measure space and $\X=\mathcal{M}(\Sigma,\mathcal{F})$ denote the Banach space of all bounded signed measures over $(\Sigma,\mathcal{F})$ endowed with the total variation norm:
$$\|\mu\|=|\mu|(\Sigma), \qquad \forall \mu \in \mathcal{M}.$$
We recall here that $\X=\mathcal{M}(\Sigma,\mathcal{F})$ is a $AL$-space \cite[Example 3, p. 114]{schaef} and every $\mu \in \X$ splits as $\mu=\mu_+-\mu\-$ where $\mu_\pm \in \X_+$ and $|\mu|=\mu_++\mu_-.$  Given two measures $\mu$ and $\nu$ of $\X$, we shall denote $\nu \prec \mu$ if $\nu$ is absolutely continuous with respect to $|\mu|.$ Using the terminology of \cite{aks}, we shall say that a closed subspace $\mathscr{A}$ of $\X=\mathcal{M}(\Sigma,\mathcal{F}) $ is a $M$-ideal if, for any $\mu \in \mathscr{A}$ and any $\nu \in \X$, $\nu \prec \mu$ implies $\nu \in \mathscr{A}.$ Then, one has the following
\begin{propo}\label{mesure}
A subspace $\mathscr{A}$ of $\mathcal{M}(\Sigma,\mathcal{F})$ is a $M$-ideal of $\mathcal{M}$ if and only if $\mathscr{A}$ is a closed and positively generated order ideal of $\mathcal{M}(\Sigma,\mathcal{F})$.
\end{propo}\begin{proof} Let us first assume that $\mathscr{A}$ is a closed and positively generated order ideal of $\X$ and let $\mu \in
\mathscr{A}$ and $\nu \in \X$ such that $\nu \prec \mu .$ From  Radon-Nikodym Theorem, there is some $h\in L^{1}(\Sigma ,\mathcal{F})%
,\d\left| \mu \right| )$ such that $ \nu =h \left| \mu \right| $. Thus, $ \left|
\nu \right| =\left| h\right| \left| \mu \right|$ and
\begin{equation*}
\lim_{n\rightarrow \infty }\left\| |\nu  | -\beta _{n}\right\|=0
\end{equation*}
where $\beta _{n}:=\left(| h | \wedge n\right)\left| \mu \right|$. Indeed $ \beta _{n}\leq \left| \nu \right|$ for any $n \in \mathbb{N}$ and
\begin{equation*}
\left\| \left| \nu \right| -\beta _{n}\right\| = \left| \nu \right|(\Sigma) -\beta
_{n}(\Sigma )=\int_\Sigma \left[ \left| h\right| - (\left| h\right|  \wedge n)\right]
\d\left| \mu \right|
\end{equation*}
goes to zero as $n \to \infty$ according to the dominated convergence theorem. Now, $\beta _{n}\leq n\left| \mu \right| $ with $|\mu| \in \mathscr{A}$ (see Remark \ref{remarkhere}) and, from the ideal property, $\beta _{n}\in \mathcal{A}$. From the closedness of $\mathscr{A}$, one gets that  $\left| \nu \right|
\in \mathcal{A}$. Since $-\left| \nu \right| \leq \nu \leq \left| \nu \right|$, one finally obtains $\nu \in
\mathcal{A}$ and $\mathscr{A}$ is a $M$-ideal. Conversely, let  $\mathscr{A}$ be a $M$-ideal. By definition, if $\mu \in \mathscr{A}$ then $\left|
\mu \right| \in \mathscr{A}$ and $\mu _{\pm}\in \mathscr{A}%
. $ In particular, $\mathscr{A}=(\mathscr{A} \cap \X_{+})-(\mathscr{A} \cap
\X_{+})$. Moreover, since $0\leq \mu \leq \nu \implies \mu \prec \nu $, one sees that  $\mathscr{A}\cap \X_{+}$ is an hereditary subcone of $\X_+$ and $\mathscr{A}$ is an order ideal of $\X$ according to \cite[Lemma 2]{robinson}.
\end{proof} One deduces from this the following which allows to give a complete description of the state $\mu$ leading to a dishonest trajectory (see Remark \ref{hdd}):
\begin{propo} Under the assumptions of Theorem \ref{structure} with $\X=\mathcal{M}(\Sigma,\mathcal{F})$, one has $\mathscr{H}$ is a $M$-ideal of $\X$ and $\X=\mathscr{H} \oplus \mathscr{H}_d$ where
\begin{equation}\label{hdperp}\mathscr{H}_d=\{\mu \in \X=\mathcal{M}(\Sigma,\mathcal{F})\;\, \text{ such that }  \nu \prec \mu  \text{ and } \nu \in \mathscr{H} \implies \nu =0 \;\}.\end{equation}
\end{propo}
\begin{proof} We saw in Theorem \ref{structure} that $\mathscr{H}$ is a closed
lattice ideal of $\X$. In particular, one can define a band projection $\mathbf{P}$
onto $\mathscr{H}$   such that $\mathscr{H}=\mathbf{P}\X$ and the
disjoint complement $\mathscr{H}_d$  of $\mathscr{H}$  given by
$\mathscr{H}_d=(I-\mathbf{P})\X$ are such that $\X=\mathscr{H}
\oplus \mathscr{H}_d$ \cite{schaef}. Since, according to Prop. \ref{mesure}, $\mathscr{H}$ is a $M$-ideal of $\X$, one deduces from \cite{aks} that
$\mathscr{H}_d=\mathscr{H}^{\perp}$ where $\mathscr{H}^{\perp}$ is given by \eqref{hdperp}.
\end{proof}
\subsubsection*{Example 2: The space of trace class operators} We assume here that $\X=\mathscr{T}_s(\mathfrak{h})$ is the Banach space of all linear self-adjoint trace class operators on some separable Hilbert space $\mathfrak{h}$ endowed with the trace norm $\|\varrho\|=\mathrm{Trace}[\,|\varrho|\,]$ for any $\varrho\in \X$ (see \cite{class} for details). The scalar product of $\mathfrak{h}$ shall be denoted by $(\cdot,\cdot)$. Under the assumptions of the present section, one deduces from \cite[Theorem 5]{class} that, for any $\l >0$, there exists $\beta_\l \in \mathscr{L}^+_s(\mathfrak{h})$ such that
$$\la \Xi_\l,u\ra=\mathrm{Trace}[\beta_\l \varrho] \qquad \forall \varrho \in \X_+$$
where $\mathscr{L}^+_s(\mathfrak{h})$ is the space of all positive bounded self-adjoint operators on $\mathfrak{h}$. One has the following
\begin{theo} The null space of $\beta_\l$ is independent of $\l$ and $$\mathcal{H}=\left\{\varrho \in \mathfrak{X}_+\,;\, \varrho=\mathbf{P}  \varrho=\varrho\mathbf{P}\,\right\}=\left\{\varrho \in \mathfrak{X}_+\,;\, \mathbf{Q}  \varrho=\varrho\mathbf{Q}=0\,\right\}$$ where $\mathbf{P}$ is the  projection of $\mathfrak{h}$ onto $\mathrm{Null}(\beta_\lambda)$ while $\mathbf{Q}=\mathbf{Id}_{\mathfrak{h}}-\mathbf{P}.$
\end{theo}
\begin{proof} Let $\l >0$ be fixed. According to Theorem 3.15, $\mathcal{H}$ is a closed hereditary subcone of $\mathfrak{X}_+$. On the other hand, closed hereditary cones of $\X$ are characterized in [10, Lemma 3.2, P. 54-55] which tells us that the set
$$\mathfrak{h}_0=\left\{h \in \mathfrak{h}\,;\,|h\rangle\langle h| \in \mathcal{H}\right\}$$
is a closed linear subspace\footnote{Notice that, in [10, Lemma 3.2, P. 54-55], Davies calls ideal what we call  closed hereditary subcone} of $\mathfrak{h}$ and
$$\mathcal{H}=\left\{\varrho \in \mathfrak{X}_+\,;\, \varrho=\mathbf{P}  \varrho=\varrho\mathbf{P}\,\right\}$$
where $\mathbf{P}$ is the orthogonal projection of $\mathfrak{h}$ onto $\mathfrak{h}_0$ while $|h\rangle\langle h|$ denotes the one-dimensional trace class operator : $\:x \mapsto (x,h)h.$ The proof consists in showing that $\mathrm{Null}(\beta_\l)=\mathfrak{h}_0$ for any $\l >0$. First, let $h \in \mathfrak{h}_0$, $h \neq 0$ and let $\varrho=|h\rangle\langle h|$. For any  orthonormal basis $(e_n)_n$ of $\mathfrak{h}$ we have  \begin{equation*}\begin{split}\mathrm{Trace}[\beta_\lambda \varrho]&=\sum_n (\beta_\lambda \varrho (e_n),e_n)=\sum_n (\varrho e_n, \beta_\lambda(e_n))\\
 &=\sum_n (h,e_n)\,(h,\beta_\lambda(e_n)) =\sum_n(h,e_n)\,(\beta_\lambda(h),e_n).\end{split}\end{equation*}
Choosing in particular a basis $(e_n)_n$ with $e_0=h/\|h\|$, one gets that
$$\mathrm{Trace}[\beta_\lambda \varrho]=0 \Longleftrightarrow (\beta_\lambda(h),h)=0 \Longleftrightarrow h \in \mathrm{Null}(\beta_\lambda)$$
since $\beta_\l \geq 0.$ This proves that $\mathfrak{h}_0=\mathrm{Null}(\beta_\l)$ which, in particular, turns out to be independent of $\l >0$. Finally, since $\mathbf{PQ}=0$ and $\mathbf{P+Q}=\mathbf{Id}$, we see that $\varrho= \mathbf{P}  \varrho=\varrho\mathbf{P}$ amounts to $\mathbf{Q}  \varrho=\varrho\,\mathbf{Q}=0$. This is equivalent to $\mathbf{Q}  \varrho\, \mathbf{Q}=0.$
\end{proof}
This allows to provide a full characterization of $\mathscr{H}$:
\begin{cor} One has $\mathscr{H}=\mathcal{H}-\mathcal{H}=\left\{\varrho \in \mathfrak{X}\,;\, \varrho=\mathbf{P}  \varrho=\varrho\mathbf{P}\,\right\}.$
\end{cor}
\begin{proof} The fact that $\mathscr{H} \subset \left\{\varrho \in \mathfrak{X}\,;\, \varrho=\mathbf{P}  \varrho=\varrho\mathbf{P}\,\right\}$ is clear. Conversely, let $\varrho \in \X$ be such that $\varrho= \mathbf{P}  \varrho=\varrho\,\mathbf{P}$. Since $\varrho \in \mathscr{T}_s(\mathfrak{h})$, one has
$$\varrho=\sum_n \alpha_n |e_n \rangle \langle e_n|$$
where $(e_n)_n$ is an orthonormal basis  of $\mathfrak{h}$ made of eigenvectors of $\varrho$ associated to the real eigenvalues $(\alpha_n)_n$, i.e. $\varrho(h)=\sum_n \alpha_n (h,e_n)e_n$ for any $h \in \mathfrak{h}.$ Since $\varrho=\varrho\,\mathbf{P}$, one has
$$\varrho(h)=\sum_n \alpha_n (h,e_n)e_n=\sum_n \alpha_n (\mathbf{P}h,e_n)e_n=\sum_n \alpha_n (h,\mathbf{P}e_n)e_n \qquad \qquad  \forall h \in \mathfrak{h}$$
while, since $\mathbf{P}\varrho=\varrho$, one has $\varrho(h)=\sum_n \alpha_n (h,\mathbf{P}e_n)\mathbf{P}e_n$ for any $h \in \mathfrak{h}$. In particular,
$$\varrho=\sum_n \alpha_n |\mathbf{P}e_n \rangle \langle \mathbf{P}e_n|.$$
As we saw in the proof of the above theorem,  $|\mathbf{P}e_n \rangle \langle \mathbf{P}e_n| \in \mathcal{H}$ for any $n \in \mathbb{N}$ so that, writing $\alpha_n=\alpha_n^+-\alpha_n^-$ with $\alpha_n^\pm \geq 0$, we see that $\varrho=\varrho^+ - \varrho^-$ with $\varrho^\pm \in \mathcal{H}.$
\end{proof}
\subsection{Sufficient conditions of honesty} We provide here sufficient conditions of honesty based on the above
Theorem \ref{char} and on a new derivation of the functional $\bl$

\begin{theo}\label{spect} For any $\l >0$, let $(\bi_n(\l))_n \subset \X^\star$ be defined
inductively by
$$\bi_{n+1}(\l)=\left[\B(\l-\A)^{-1}\right]^\star\bi_n(\l),\qquad
\bi_0(\l)=\mathbf{\Psi}$$ where  we recall that $\mathbf{\Psi}$ is
the positive functional defined in \eqref{psi}. Then,
$(\bi_n(\l))_n$ is nonincreasing and converges in the weak-$\star$
topology of $\X$ to $\bi(\l)$ such that
\begin{equation}\label{betastar}\left[\B(\l-\A)^{-1}\right]^\star\bi(\l)=\bi(\l).\end{equation}
 Moreover,
$\bi(\l)=\bl$ for all $\l >0 $  and $\bl $ is the maximal element of
$\{\psi \in \X^\star,\,\psi \leq \mathbf{\Psi}\}$ satisfying
\eqref{betastar}.\end{theo}

{\begin{proof} It is clear that $\left[\B(\l-\A)^{-1}\right]^\star$ is a
positive contraction in $\X^\star$. Then, for all $\bi \in
\X_+^\star$ with $\|\bi\| \leq 1$,
$$\left\|\left[\B(\l-\A)^{-1}\right]^\star\bi\right\|\leq 1$$
or, in an equivalent way,
$$\la
\left[\B(\l-\A)^{-1}\right]^\star\bi,u\ra \leq \|u\|=\la
\mathbf{\Psi},u\ra,\qquad \forall  u \in \X_+,$$ i.e.
$\mathbf{\Psi}-\left[\B(\l-\A)^{-1}\right]^\star\bi$ is an
element of the positive cone of $\X^\star$. Actually, it is
straightforward to see that, for any given $u \in \X_+$, the
sequence $\left(\la \bi_n(\l), u \ra \right)_n$ is bounded and
nonincreasing in $\mathbb{R}_+$. This means that $(\bi_n(\l))_n$
converges in the weak-$\star$ topology to some $\bi(\l) \leq
\mathbf{\Psi}$. Let $u \in \X_+$ be given. Then,
$$\la \psi_{n+1}(\l),u \ra
=\la \left[\B(\l-\A)^{-1}\right]^\star\bi_n(\l),u \ra =\la
\bi_n(\l),\B(\l-\A)^{-1}u\ra$$ so, letting $n \to \infty$,
$$\la \bi (\l),u\ra=\la \bi(\l) ,\B(\l-\A)^{-1}u\ra$$
which shows \eqref{betastar}. Now, since
\begin{equation*}\begin{split}
\la \bl,u\ra&=\lim_{n \to \infty} \la \mathbf{\Psi}, \bigg[
\B(\l-\A)^{-1}\bigg]^{n+1}u\ra\\ &=\lim_{n \to \infty}\la
\bigg(\left[\B(\l-\A)^{-1}\right]^{n+1}\bigg)^\star \mathbf{\Psi},u \ra\\
&=\lim_{n\to \infty}\la\bi_{n+1}(\l),u \ra=\la \bi (\l),u\ra
\end{split}\end{equation*}
one sees that $\bi (\l)=\bl.$ Let us now prove that $\bi(\l)=\bl$ is
the maximal element of $\{\psi \in \X^\star,\,0 \leq \psi \leq
\mathbf{\Psi}\}$ satisfying \eqref{betastar} $(\l >0)$. To do so,
let $\psi$ be in the positive cone of $\X^\star$, $\psi \leq
\mathbf{\Psi}$ be such that
$\left[\B(\l-\A)^{-1}\right]^\star\psi=\psi$. Then,
$$\psi=\bigg(\left[\B(\l-\A)^{-1}\right]^\star\bigg)^n \psi \leq
 \bigg(\left[\B(\l-\A)^{-1}\right]^\star\bigg)^n \mathbf{\Psi}$$
 which proves, letting $n$ go to infinity, that $\psi \leq \bl$.
\end{proof}}

 As a consequence,
one has
\begin{cor}
Assume there exists $\l >0$ such that  $\B(\l-\A)^{-1}$ is
irreducible. Then, the whole
semigroup $\Tt$ is honest if and only if there is some $u \in \X_+$,
$u \neq 0$, for which the trajectory $(\T(t)u)_{t \geq 0}$ is
honest.
\end{cor}
{\begin{proof} We give two proofs of this result. The first one uses
Theorem \ref{theo:irr} and the second one the spectral
interpretation of the functional $\bl$.

\indent {\it Proof 1}. Let $u \in \X_+ \setminus \{0\}$ and $\omega
\in \X^\star_+ \setminus \{0\}$. Then, $(\l-\A^\star)^{-1}\omega \in
\X^\star_+\setminus \{0\}$ and
\begin{equation*}\begin{split}
\la \omega, (\l-\G)^{-1}u \ra &=\sum_{k=0}^\infty \la \omega,
(\l-\A)^{-1}\left[\B(\l-\A)^{-1}\right]^k u\ra\\
&=\sum_{k=0}^\infty \la (\l-\A^\star)^{-1}\omega,
\left[\B(\l-\A)^{-1}\right]^k u \ra >0\end{split}
\end{equation*}
where we used the fact that there exists $k_0>0$ such that
$$\la
(\l-\A^\star)^{-1}\omega,
\left[\B(\l-\A)^{-1}\right]^{k_0} u \ra
>0.$$ One obtains then that $\la \omega, (\l-\G)^{-1}u \ra >0$ for
any $\omega \in \X_+^\star \setminus \{0\}$, i.e. $(\l-\G)^{-1}u$ is
quasi-interior for any $u\in \X_+ \setminus \{0\}$. Thus, $\Tt$ is
irreducible and Theorem \ref{theo:irr} leads to the conclusion.

{\it Proof 2}. Let $\B(\l-\A)^{-1}$ be irreducible and
assume there exists some honest trajectory $(\T(t)u)_{t \geq 0}$
with $u \in \X_+ \setminus \{0\}.$ Then, from Theorem \ref{char},
$\la \bl,u\ra=0$. Assume that $\bl \neq 0$. Then, for any $z \in
\X_+ \setminus \{0\}$, there exists an integer $n \geq 0$ such that
$$\la \bl, \left[\B(\l-\A)^{-1}\right]^n z\ra >0.$$
According to Theorem \ref{spect}, it is clear that
$$\la \bl,z\ra=\la \left(\left[\B(\l-\A)^{-1}\right]^\star\right)^n
\bl, z\ra=\la \bl, \left[\B(\l-\A)^{-1}\right]^n z\ra$$
i.e. $\la \bl, z\ra >0$ for any $z \in \X_+\setminus \{0\}$. This is
a contradiction and, necessarily, $\bl = 0$. Thus, the whole
semigroup $\Tt$ is honest.
\end{proof}}

We end this section with two practical sufficient conditions
ensuring the existence of honest trajectories:
\begin{theo}\label{theo:sousol} Let $\l >0$ and $u \in \X_+$ be such that
\begin{equation}\label{soussol}
\B(\l-\A)^{-1}u \leq u,\end{equation} then the trajectory
$(\T(t)u)_{t \geq 0}$ is honest.\end{theo} {\begin{proof} Since
$\B(\l-\A)^{-1}$ is positive, our assumption
\eqref{soussol} implies that the sequence $
([\B(\l-\A)^{-1}]^n u )_n$ is nonincreasing in $\X$ and
$$\left[\B(\l-\A)^{-1} \right]^nu \leq u,\qquad \forall n \geq 1.$$
Therefore the whole sequence
$\left([\B(\l-\A)^{-1}]^n u\right)_n$ is convergent in
$\X$ which ends the proof because of Theorem \ref{equivalence}.
\end{proof}}

This provides another honesty criterion in terms of sub-eigenvalues of
$\A+\B$.

\begin{cor} Assume that there exists $\l >0$ and $u \in \D(\A)_+$
such that $(\A+\B)u \leq \l u$, Then, $(\T(t)u)_{t \geq
0}$ is honest.
\end{cor}
\begin{proof} Define $z=(\l-\A)u$. One has $z \geq \B u \geq 0$ and
$z$ satisfies \eqref{soussol}. The trajectory $(\T(t)z)_{t \geq 0}$
is therefore honest from Theorem \ref{theo:sousol}. Defining
$v=(\l-\G)^{-1}z$, one has also that $(\T(t)v)_{t \geq 0}$ is honest
(see Proposition \ref{closed}). Since $0 \leq u=(\l-\A)^{-1}z \leq
v$,  $(\T(t)u)_{t \geq 0}$ is honest since $\mathcal{H}$ is a closed
hereditary subcone of $\X_+$ (see Theorem \ref{structure}).
\end{proof}

\subsection{Instantaneous dishonesty}\label{sub:instant} According to Definition \ref{defi:hon}, if a trajectory $(\T(t)u)_{t \geq 0}$ is
 not honest, then  there exists $t_0 \geq 0$ such that
\begin{equation}\label{dishon} \|\T(t_0)u\| <
\|u\|-\overline{\cc}\bigg(\int_0^{t_0} \T(s)u \d
s\bigg)\end{equation}This suggests to introduce the following mass
loss functional
$$\Delta_u (t)=\|\T(t)u\|-\|u\|+\overline{\cc}\bigg(\int_0^t \T(s)u\d
s\bigg),\qquad t \geq 0.$$ One has the following property:
\begin{lemme}
For any $u \in \X_+$, the mapping $t \geq 0 \mapsto \Delta_u(t)$ is
nonincreasing.
\end{lemme}
{\begin{proof} Let $t_2 \geq t_1 \geq 0$ be fixed. Then,
\begin{equation*}\begin{split}
\Delta_u(t_2)-\Delta_u(t_1)&=\|\T(t_2)u\|-\|\T(t_1)u\|+\overline{\cc}(\int_{t_1}^{t_2}\T(s)u\d
s)\\
&=\la \mathbf{\Psi}, \T(t_2)u-\T(t_1)u\ra
+\overline{\cc}(\int_{t_1}^{t_2}\T(s)u\d s).
\end{split}\end{equation*}
Since $\T(t_2)u-\T(t_1)u=\G\ds \int_{t_1}^{t_2}\T(s)u\d s$, one sees
that
$$\Delta_u(t_2)-\Delta_u(t_1)=\overline{\cc}(\int_{t_1}^{t_2}\T(s)u\d
s)-\cc_0(\int_{t_1}^{t_2}\T(s)u\d s) \geq 0,$$ since
$\overline{\cc}$ always dominate $\cc_0$.
\end{proof}}

\begin{lemme}\label{24} Let $u \in \X_+$. If the trajectory $(\T(t)u)_{t \geq
0}$ is dishonest, then there exists $t_0 >0$ such that $\Delta_u(t)
<0$ for any $t >t_0$ and $\Delta_v(t) <0$ for any $t >0$ where
$v=\T(t_0)u$.
\end{lemme}
{\begin{proof} By definition of dishonest trajectory and since
$\Delta_u(t) \leq 0$ for any $t \geq 0$, one has
$$t_0:=\inf\{t >0\,\text{ such that } \Delta_u(t) <0\}$$
is well-defined. Since $\Delta_u(\cdot)$ is nonincreasing, one has
$\Delta_u(t) < 0$ for any $t > t_0$. Moreover, since the mapping $t
\mapsto \Delta_u(t) \in (-\infty,0]$ is clearly continuous, one has
$\Delta_u(t)=0$ for any $t \in [0,t_0].$ Set $v=\T(t_0)u$. For any
$t >0$, since $\Delta_u(t+t_0) <0$ one has
$$\|\T(t)v\|=\|\T(t+t_0)u\| < \|u\|-\overline{\cc}(\int_0^{t+t_0}\T(s)u
\d s)$$ while the identity $\Delta_u(t_0)=0$ reads
$\|u\|=\|v\|+\overline{\cc}(\int_0^{t_0}\T(s)u\d s).$ Consequently,
\begin{equation*}\begin{split}
\|\T(t)v\| &< \|\T(t_0)u\|+\overline{\cc}\left(\int_0^{t_0}\T(s)u\d
s\right)-\overline{\cc}\left(\int_0^{t+t_0}\T(s)u \d s\right)\\
&=\|v\|-\overline{\cc}\left(\int_{t_0}^{t+t_0}\T(s)u\d
s\right)=\|v\|-\overline{\cc}\left(\int_{0}^t \T(r)v\d r\right)
\end{split}\end{equation*}
i.e. $\Delta_v(t) < 0$ for all $t >0$.\end{proof}}

To summarize, when the semigroup $\Tt$ is dishonest, that is, if
some trajectory $(\T(t)u)_{t \geq 0}$ is not honest, then it is
possible to find some $z \in \X_+ \setminus\{0\}$ such that the
trajectory emanating from $z$ is \textit{ instantaneously} {dishonest}, i.e. $\Delta_z(t) < 0$ for any $t >0$. In particular,
whenever
\begin{equation*} \la \mathbf{\Psi}, (\A+{\B})u \ra = 0
\qquad \forall u \in \D_+,\end{equation*} the semigroup $\Tt$ is
dishonest if and only if there exists some $z \in \X_+
\setminus\{0\}$ such that
$$\|\T(t)z\| < \|z\|, \qquad \forall t >0.$$
\begin{theo}
Assume that $\la \mathbf{\Psi}, (\A+{\B})u \ra = 0$ for any
$u \in \D_+.$ If $\G \neq \overline{\A+\B}$ then
$$\|\T(t)u\| < \|u\|, \qquad t
>0$$
for any quasi-interior $u \in \X_+$.
\end{theo}
{\begin{proof} If $\Tt$ is dishonest, then, according to Lemma
\ref{24}, there exists $z \in \X_+ \setminus \{0\}$ such that
$\|\T(t)z\| < \|z\|$ for all $t >0$. In particular,
$$\la \mathbf{\Psi}, \T(t)z-z\ra <0, \qquad \forall t >0.$$
Define $\mathcal{Z}_t:=\mathbf{\Psi}- \T^\star(t) \mathbf{\Psi} \in
\X^\star,$ for any $t >0$ where $(\T^\star(t))_{t \geq 0}$ is the
dual contractions semigroup of $\Tt$. Since $\la
\mathbf{\Psi},\T(t)u-u\ra \leq 0$ for any $u \in \X_+$,   $
\mathcal{Z}_t$ belongs to the positive cone $\X_+^\star$ of
$\X^\star$ for any $t \geq 0$ while
$$\la \mathcal{Z}_t, z \ra > 0,\qquad \forall t >0.$$
Therefore, $\mathcal{Z}_t$ belongs to $\X^\star_+ \setminus \{0\}$
for any $t >0$. Therefore, for any quasi-interior $u \in \X_+$ one
has
$$\la \mathcal{Z}_t, u \ra > 0,\qquad \forall t >0.$$
This proves the result.
\end{proof}}

\begin{nb} Whenever $\X$ \textbf{is an $AL$-space}, it is possible to prove a more general result of immediate
dishonesty by resuming in a straightforward way the arguments of
\cite[Corollary 2.12]{mkvo}. Precisely, recall  that, if $\X$ is an
$AL$-space, $\mathscr{H}$ is a projection band of $\X$ (see Prop.
\ref{propo:von}) and  let $\mathbf{P}$ be the band projection
onto $\mathscr{H}$. Then, one can prove the
following: let us assume that $\Tt$ is not honest and let $u \in
\X_+$ be such that $v=(I-\mathbf{P})u$ is a quasi-interior element
of the disjoint complement of $\mathscr{H}$. Then, the trajectory
$(\T(t)u)_{t \geq 0}$ is \textbf{immediately dishonest}, i.e.
 $\|\T(t)u\| < \|u\| - \overline{\cc} (\int_0^t \T(s)u \d s )$ for any $ t >0.$ However, from the technical point of view, the formal arguments need the use of the concept of local honesty as in \cite{mkvo}.\end{nb}

\section{On honesty theory:  Dyson-Phillips
approach}
We establish here an alternative of concept of honesty of the
trajectory in terms of  the Dyson-Phillips iterated   defined by
\eqref{dysonphi2}. To do so, we have first to investigate several fine properties of these iteration terms.
\subsection{Fine properties of the Dyson-Phillips iterations} The various terms of the Dyson-Phillips series \eqref{dysonphi2} enjoy  the following properties:
\begin{propo}\label{luisa2} For any $n \in \mathbb{N},$ $n \geq 1$, the Dyson-Phillips
iterated defined in \eqref{dysonphi2} satisfy:
\begin{enumerate}
\item For any $u \in \D(\A)$, the mapping $t \in (0,\infty) \longmapsto \T_n(t)u$ is
continuously differentiable with
$$\dfrac{\d}{\d t}\T_n(t)u=\T_n(t)\A u +
\T_{n-1}(t) \B u.$$
\item For any $u \in \D(\A)$,  $\T_n(t)u \in \D(\A)$, the mapping $t \in (0,\infty) \longmapsto \A\T_n(t)u$ is continuous and
$$\A\T_n(t)u=\T_n(t)\A u +\T_{n-1}(t)\B u-\B \T_{n-1}(t) u$$
\item For any $u \in \X$ and any $t \geq 0$, $\ds\int_0^t \T_n(s)u \d s
\in \D(\A)$, the mapping $t \in (0,\infty) \longmapsto \A \int_0^t
\T_n(s)u \d s$ is continuous with \begin{equation}\label{form3}\A
\int_0^t  \T_n(s)u \d s= \T_n(t)u- \B\int_0^t \T_{n-1}(s)u
\d s.\end{equation}
\item For any $u \in \X_+$ and any $t \geq 0$,
\begin{equation}\label{form4}\la \mathbf{\Psi},\B\int_0^t \T_n(s)u\d s \ra \leq -\la
\mathbf{\Psi}, \T_n(t)u\ra+\la \mathbf{\Psi}, \B\int_0^t
\T_{n-1}(s)u \d s \ra.\end{equation}
\item For any $u \in \X$,  and $\l >0$, the limit
$$\lim_{t \to \infty} \int_0^t \exp(-\l s)\T_n(s)u\d s=:\int_0^\infty
\exp(-\l s)\T_n(s)u\d s$$ exists in the graph norm of $\A$ and
\begin{equation}\label{mus4}
\left(\l-\A\right) \int_0^\infty \exp(-\l s)\T_n(s)u\d s=  \B \int_0^\infty \exp(-\l s)\T_{n-1}(s)u
\d s.
\end{equation}
\end{enumerate}
\end{propo}
\begin{nb} Notice that, in Eq. \eqref{mus4}, $\B \int_0^\infty \exp(-\l s)\T_{n-1}(s)u
\d s$ is well-defined since $\B$ is $\A$-bounded and the integral converges in the graph norm of $\A$. Moreover, it is easily deduced from \eqref{mus4} that
\begin{equation}\label{blan}
\B\int_0^\infty \exp(-\l s) \T_n(s)u\d s =
\left[\B(\l-\A)^{-1}\right]^{n+1} u, \qquad \forall u \in \X,
\:n \geq 1.\end{equation}
\end{nb}
\begin{proof}  We first recall that
the formula \eqref{dysonphi2} reads on $\D(\A)$ as:
$$\T_{n+1}(t)u=\int_0^t \T_n(t-s)\B\u(s)u\d s, \qquad
\forall u \in \D(\A), \qquad t \geq 0,  \qquad n \in \mathbb{N}.$$
Then
$$h^{-1}\T_{n+1}(h)u=h^{-1}\int_0^h \T_n (h-s)\B \u(s)u
\longrightarrow \T_n(0)\B\u(0)u \qquad \text{ as } \quad h
\to 0^+$$ because the mapping $(s,h) \mapsto \T_n (h-s)\B \u(s)u$ is strongly continuous on $\{(s,h) \in \mathbb{R}_+ \times
\mathbb{R}_+\,;\,0 \leq s \leq h\}.$  Since $\T_{0}(0)
=\u_0(0)=\mathrm{Id}$ while $\T_n(0)=0$ for any $n \geq 1$, we see
that
\begin{equation}\label{t1j} \lim_{h \to 0} h^{-1}\T_{n+1}(h)u=\begin{cases} \B
u \qquad \qquad  &\text{ when } n=0\\
 0 \qquad  \qquad  &\text{ when } n \geq 1, \qquad \forall u \in
\D(\A).\end{cases}\end{equation}

\textit{(1)} Let $n \geq 1$ be fixed. Let $u \in \D(\A)$ and  $t,h
>0$ be fixed. One deduces from
\eqref{luisanotes2} that, given $t \geq 0$ and $h >0$,
\begin{equation*}\label{0.6}\begin{split}
\T_{n}(t+h)u-\T_{n}(t)u&=\sum_{k=0}^{n-1}\T_k(t)\T_{n-k}(h)u +
\T_{n}(t)(\T_0(h)u-u)\\
&=\sum_{k=1}^{n}\T_{n-k}(t)\T_{k}(h)u+ \T_{n}(t)(\T_0(h)u-u).
\end{split}
\end{equation*} Thus,
$$h^{-1}(\T_{n}(t+h)u-\T_{n}(t)u)=\sum_{k=1}^{n}\T_{n-k}(t)\left(\frac{1}{h}\T_k(h)u\right)
+\T_{n}(t)\frac{\T_0(h)u-u}{h} $$ which yields, since $u \in
\D(\A)$,
$$\lim_{h \to 0^+}\frac{\T_{n}(t+h)u-\T_{n}(t)u}{h}=\T_{n-1}(t)\B
u+\T_{n}(t)\A u.$$ Similarly, it is easy to prove
 that, for any $t >0$ and any $0 < h < t$,
\begin{equation}\label{0.7}
\T_{n}(t)u-\T_{n}(t-h)u=\sum_{k=1}^{n}\T_{n-k}(t-h)\T_k(h)u+
\T_{n}(t-h)(\T_0(h)u-u)\end{equation} and therefore
$$\lim_{h \to 0^+}\frac{\T_{n}(t)u-\T_{n}(t-h)u}{h}=\T_{n-1}(t)\B
u+\T_{n}(t)\A u.$$ Since, for any $u \in \D(\A)$, the mapping $t
\mapsto \T_{n-1}(t)\B u+\T_{n}(t)\A u$ is continuous (see Remark \ref{remark24}), property
\textit{(1)} holds true.\\

\textit{(2)}  Let $u \in \D(\A)$. It is clear that the two properties
$$\T_k(t)u \in \D(\A) \qquad \text{ and } \quad t \geq 0 \mapsto \A\T_k(t)u \text{ is continuous}$$
hold true for $k=0$. Let $n \geq 1$ be fixed and assume the above properties hold true for any $k\leq n$ and prove they still hold for $k=n+1$. For any $t,h >0$, Eq.
\eqref{luisanotes2} yields
$$\T_{n+1}(t+h)u=\T_{n+1}(h+t)u=\sum_{k=0}^{n+1}\T_k(h)\T_{n+1-k}(t)u$$
so that
\begin{equation}\label{0.9}
\T_0(h)\T_{n+1}(t)u-\T_{n+1}(t)u=(\T_{n+1}(t+h)u-\T_{n+1}(t)u)-\sum_{k=1}^{n+1}\T_k(h)\T_{n+1-k}(t)u.\end{equation}
Assume now $u \in \D(\A)$, by virtue of point \textit{(1)} and
\eqref{t1j}, we have
$$\lim_{h \to
0^+}\dfrac{\T_0(h)\T_{n+1}(t)u-\T_{n+1}(t)u}{h}=\T_{n}(t)\B
u+\T_{n+1}(t)\A u- \B \T_{n}(t)u.$$  This shows that
$\T_{n+1}(t)u \in \D(\A)$ for any $u \in \D(\A)$ with
$$\A \T_{n+1}(t)u = \T_{n+1}(t)\A u+\T_{n}(t)\B u -\B \T_{n}(t)u$$
and proves \textit{(2)} since the continuity of the mapping $t \geq 0 \mapsto \A \T_{n+1}(t)u$ is easy to prove.\\

\textit{(3)} The first part of point \textit{(3)} clearly holds for $n=0$. Let $u \in \X$ and $n \in \mathbb{N}$ be fixed. Assume that, for any $t \geq 0$ and any $k \leq n$,  $\ds\int_0^t \T_k(s)u \d s
\in \D(\A)$, the mapping $t \in (0,\infty) \longmapsto \A \int_0^t
\T_k(s)u \d s$ is continuous. Let us prove the result for $k=n+1.$ Let $t,h >0$. From \eqref{0.9} we have
\begin{equation*}\begin{split}
(\T_0(h)-\mathrm{Id})&\int_0^t \T_{n+1}(s)u \d s=\int_0^t
(\T_0(h)\T_{n+1}(s)u-\T_{n+1}(s)u)\d s\\
&=\int_0^t (\T_{n+1}(s+h)u-\T_{n+1}(s)u)\d s
-\sum_{k=1}^{n+1}\T_k(h)\int_0^t \T_{n+1-k}(s)u \d s\\
&=\int_{t}^{t+h}\T_{n+1}(r)u\d r -\int_0^h \T_{n+1}(r)u \d r
-\sum_{k=1}^{n+1}\T_k(h)\int_0^t \T_{n+1-k}(s)u \d s.\end{split}
\end{equation*}
Since we assumed that $\int_0^t \T_j(s)u \d s \in \D(\A) \subset
\D(\B)$ for any $0 \leq j \leq n$, we deduce immediately
\begin{equation*}
\lim_{h \to 0^+}h^{-1}(\T_0(h)-\mathrm{Id})\int_0^t \T_{n+1}(s)u \d
s=\T_{n+1}(t)u -\B\int_0^t \T_n(s)u \d s
\end{equation*}
where we used \eqref{t1j} and the fact that
$h^{-1}\int_t^{t+h}\T_{n+1}u \d s \to \T_{n+1}(t)u$ as $h \to 0^+$.
Therefore, property \textit{(3)} holds true for $n+1.$\\

\textit{(4)} Let $u \in \X$ and $t \geq 0$ be fixed. Applying
\eqref{hypprinc1} to $v=\ds \int_0^t \T_n(s)u \d s$ (which belongs to
$\D(\A)$ from point \textit{(3)}), one deduces easily \eqref{form4}
from \eqref{form3}. \\

\textit{(5)} It is clear that the definition of $\T_n(t)$ given in \eqref{dysonphi2} is equivalent to
$$\exp(-\l t) \T_{n+1}(t)u=\int_0^t \exp\left(-\l(t-s)\right)\T_n(t-s) {\B}\left[\exp\left(-\l s\right)\u(s)\right]u\d s$$
 for any $u \in \D(\A)$, $n \in
\mathbb{N}$ and any $\l >0.$ Moreover, for any $\l >0$, the
operators $\A_\l:=\A-\l$ (with domain $\D(\A)$) and $\B$ satisfy the
assumptions of Theorems \ref{kato} since
$$\la \mathbf{\Psi}, (\A_\l + \B)u \ra \leq -\l \la \mathbf{\Psi}, u
\ra \leq 0, \qquad \forall u \in \D(\A)_+,\:\l >0.$$ One sees then that
there is an extension of $\A_\l+\B$ that generates a \com
$\left(\T_\l(t)\right)_{t \geq 0}$ in $\X$. Clearly, the family
$\left(\exp(-\l t)\T_n(t)\right)_{n \in \mathbb{N}}$ is the family
of Dyson-Phillips iterated associated to $\A_\l$, $\B$ and
$\T_\l(t)$. In particular, applying Formula \eqref{form3} to
$\A_\l$, $\B$ and $\left(\exp(-\l t)\T_n(t)\right)_{n \in
\mathbb{N}}$, one gets
\begin{multline}\label{mus3}
\A\int_0^t \exp(-\l s)\T_n(s)u\d s=\exp(-\l t)\T_n(t)u + \l \int_0^t
\exp(-\l s) \T_n(s)u \d s  \\
-\B \int_0^t \exp(-\l s) \T_{n-1}(s)u \d s, \quad \forall
\l >0, \quad \forall u \in \X, \;n \geq 1.\end{multline}
Notice that, since for any $n \in \mathbb{N}$,
$\T_n(t)=\mathscr{L}^n(\u)(t)$, we
already saw in the proof of Theorem \ref{rhandi} that, for any $u
\in \X$ and any $\l
>0$, the limit $$\lim_{t \to \infty} \int_0^t \exp(-\l s) \T_{n}(s)u\d
s$$ exists in $\X$ and \begin{equation}\label{LnTn} \int_0^\infty
\exp(-\l s)\T_n(s)u\d
s=(\l-\A)^{-1}\left[\B(\l-\A)^{-1}\right]^nu, \qquad
\forall n \in \mathbb{N}.\end{equation} Now, for $n=0$, since
$$\A\int_0^t \exp(-\l s)\u(s)u \d s=\exp(-\l t)\u(t)u-u+\l \int_0^t
\exp(-\l s) \u(s)u \d s$$ one easily sees that the limit $\lim_{t
\to \infty} \A\int_0^t \exp(-\l s) \u(s)u\d s$ exists in $\X$ with
$$\lim_{t \to \infty} \A \int_0^t \exp(-\l s)\u(s)u \d s=-u +\l
\int_0^\infty \exp(-\l s)\u(s)u\d s,$$ i.e. $\ds\int_0^\infty
\exp(-\l s) \u(s)u\d s$ converges in the graph norm of $\A$. Since
$\B$ is $\A$-bounded, the limit
$$\lim_{t \to
\infty}\B\int_0^t \exp(-\l s)\u(s)u \d
s=\B\int_0^\infty \exp(-\l s) \u(s)u \d
s=\B(\l-\A)^{-1}u$$ exists in $\X$. Now, applying
\eqref{mus3} to  $n=1$, the integral $\ds \int_0^\infty \exp(-\l
s)\T_1(s)u\d s$ converges in the graph norm of $\A$ with
$$\A\int_0^\infty \exp(-\l s) \T_1(s)u\d s=\l \int_0^\infty \exp(-\l
s)\T_1(s)u\d s-\B\int_0^\infty \exp(-\l s)\u(s)u \d s$$
and, as above, since $\B$ is $\A$-bounded,
$$\lim_{t \to \infty}\B\int_0^t \exp(-\l s)\T_1(s)u\d
s=\B\int_0^\infty \exp(-\l s)\T_1(s)u \d s$$ converges in
$\X$. A simple induction leads to the result for any $n \in
\mathbb{N}$.
\end{proof}
\begin{nb} Note that $\mathcal{A}$ is closed but a priori $\mathcal{A+B}$ is not;
however for $u\in \mathcal{D}(\mathcal{A})$
\begin{equation}\label{mustrick2}\begin{split}
(\A+\B )\int_{0}^{t}\T_{k}(r)u\d r &=(\overline{\A+\B}%
)\int_{0}^{t}\T_{k}(r)u\d r=\int_{0}^{t}(\overline{\A+\B})\T_{k}(r)u\d r \\
&= \int_{0}^{t}(\overline{\A+\B})\T_{k}(r)u\d r=\int_{0}^{t}(%
\A+\B )\T_{k}(r)u\d r \\
&= \int_{0}^{t}(\A+\B)\T_{k}(r)u\d r=\int_{0}^{t}\A\T
_{k}(r)u\d r+\int_{0}^{t}\B \T_{k}(r)u\d r;\end{split}\end{equation}
in particular%
\[
\B\int_{0}^{t}\T_{k}(r)u\d r=\int_{0}^{t}\B\T_{k}(r)u\d r.\]
\end{nb}
From the above Proposition, $\lim_{t \to \infty} \B\int_t^\infty
\exp(-\l s)\T_n(s)u\d s$ converges to zero for any $n \in
\mathbb{N}$ and any $u \in \X_+$. Actually, this convergence is
uniform with respect to $n$:
\begin{propo}\label{lem:mus3} For any $\l >0$ and any $u \in \X$, one has
$$\lim_{t \to \infty} \sup_{n \in \mathbb{N}}
\left\|{\B}\int_t^\infty \exp(-\l s)\T_n(s)u\d s
\right\|=0.$$
\end{propo}
\begin{proof} The combination of \eqref{mus3} and \eqref{mus4} gives%
\begin{eqnarray*}
\A\int_{t}^{ \infty }\exp (-\lambda s)\T_{n}(s)u\d s
&=&-e^{-\lambda t}\T_{n}(t)u+\lambda \int_{t}^{ \infty }\exp
(-\lambda s)\T_{n}(s)u\d s \\
&& \phantom{++++}-\mathcal{B}\int_{t}^{\infty }\exp (-\lambda s)\T_{n-1}(s)u\d s
\end{eqnarray*}%
so that
\begin{multline*}
\la \mathbf{\Psi} ,\A\int_{t}^{ \infty }\exp (-\lambda s)\T
_{n}(s) u\ d s\ra  =\la  \mathbf{\Psi} ,\lambda \int_{t}^{ \infty }\exp (-\lambda s)\T
_{n}(s)u \d s\ra \\-\la  \mathbf{\Psi} ,e^{-\lambda t}\T_{n}(t)u\ra
-\la  \mathbf{\Psi} ,\mathcal{B}\int_{t}^{ \infty }\exp (-\lambda s)\T
_{n-1}(s)u\d s\ra .
\end{multline*}
Since $\int_{t}^{ \infty }\exp (-\lambda s)\T_{n}(s)u \d s\in \D(%
\A)_{+}$   for $u \in \X_{+}$ then by \eqref{hypprinc1}
$$
\la  \mathbf{\Psi} ,\A\int_{t}^{\infty }\exp (-\lambda s)\T
_{n}(s)u \d s\ra  \leq -\la  \mathbf{\Psi} ,\mathcal{B}\int_{t}^{ \infty }\exp
(-\lambda s)\T_{n}(s)u \d s\ra
$$
whence
\begin{eqnarray*}
&&\la  \mathbf{\Psi} ,\mathcal{B}\int_{t}^{ \infty }\exp (-\lambda s)\T
_{n}(s)u \d s\ra  +\la  \mathbf{\Psi} ,\lambda \int_{t}^{ \infty }\exp (-\lambda s)%
\T_{n}(s)u \d s\ra  \\
&\leq &\la  \mathbf{\Psi} ,e^{-\lambda t}\T_{n}(t)u \ra  +\la \mathbf{\Psi} ,%
\mathcal{B}\int_{t}^{ \infty }\exp (-\lambda s)\T
_{n-1}(s)u \d s\ra.
\end{eqnarray*}%
In particular for all $n$%
\[
\la \mathbf{ \Psi} ,\mathcal{B}\int_{t}^{ \infty }\exp (-\lambda s)\T%
_{n}(s)u\d s\ra  \leq \la \mathbf{ \Psi} ,e^{-\lambda t}\T%
_{n}(t)u\ra  +\la \mathbf{ \Psi} ,\mathcal{B}\int_{t}^{ \infty }\exp (-\lambda
s)\T_{n-1}(s)u \d s\ra
\]%
and it follows by induction that%
\begin{multline*}
\la \mathbf{ \Psi} ,\mathcal{B}\int_{t}^{ \infty }\exp (-\lambda s)\T
_{n}(s)u \d s\ra  \leq \sum_{j=1}^{n}\la  \mathbf{ \Psi} ,e^{-\lambda t}\T
_{j}(t)u\ra  +\\
\la \mathbf{ \Psi} ,\mathcal{B}\int_{t}^{ \infty }\exp (-\lambda
s)\T_{0}(s)u\d s\ra  \\
\leq \la \mathbf{\Psi} ,e^{-\lambda t}\T(t)u\ra  +\la \mathbf{ \Psi} ,%
\mathcal{B}\int_{t}^{ \infty }\exp (-\lambda s)\T_{0}(s)u\d s\ra
\end{multline*}
and then%
\[
\left\Vert \mathcal{B}\int_{t}^{ \infty }\exp (-\lambda s)\T%
_{n}(s)u \d s\right\Vert \leq e^{-\lambda t}\left\Vert u\right\Vert +\left\Vert
\mathcal{B}\int_{t}^{ \infty }\exp (-\lambda s)\T
_{0}(s)u \d s\right\Vert
\]%
which ends the proof since $\X =\X_{+}-\X_{+}$. \end{proof}
\subsection{A new functional} While, in Section 3, we introduced a functional $\overline{\cc}$ related to $\cc$ through the resolvent $(\l-\A)^{-1}$, we introduce here a new functional $\widehat{\cc}$ constructed through the  Dyson-Phillips iteration terms:
\begin{propo}\label{hatccdef} Under the assumption of Theorem
\ref{kato},  for any $v \in \D(\G)$, there exists
\begin{equation}\label{hatcc}\lim_{t \to 0^+}\dfrac{1}{t}\sum_{n=0}^\infty \cc\left(\int_0^t \T_n(s)v\d s\right)=:\widehat{\cc}(v)\end{equation}
with $|\widehat{\cc}(v)| \leq 4M\left(\|v\|+\|\G v\|\right)$. Furthermore, for $v \in \D(\G)_+$, $\widehat{\cc}(v) \leq \cc_0(v) \leq \|\G v\|.$
\end{propo}
\begin{proof} First, one notices that, for any $u \in \X_+$, $n \in \mathbb{N}$ and any $t >0$, one has
$$\sum_{k=0}^n\cc\left(\int_0^t \T_k(s)u \d s\right) \leq \cc_0\left(\int_0^t \T (s)u \d s\right)=-\la \mathbf{\Psi}, \G\int_0^t \T (s)u \d s\ra.$$
In particular, the series $\sum_{k=0}^\infty\cc\left(\int_0^t \T_k(s)u \d s\right)$ converges with
\begin{equation}\label{lui1}\sum_{k=0}^\infty\cc\left(\int_0^t \T_k(s)u \d s\right) \leq -\la \mathbf{\Psi}, \G\int_0^t \T (s)u \d s\ra \leq \|u\|.\end{equation}
Now, for any integers $0<n_1<n_2<n_3$, since, for any $s,r \geq 0$
\begin{multline*}\sum_{k=0}^{n_1}\T_k(s) \left(\sum_{p=0}^{n_2} \T_p(r)u\right) \leq \sum_{k=0}^{2n_2}  \sum_{p=0}^{2n_2-k}\T_k(s)\T_p(r)u=\sum_{k=0}^{2n_2}\T_k(s+r)u\\ \leq \sum_{k=0}^{2n_2}\T_k(s)\left( \sum_{p=0}^{2n_3} \T_p(r)u\right), \qquad \forall u \in \X_+\end{multline*}
we get, for any $t,\,\tau >0$
\begin{multline*}
\sum_{k=0}^{n_1}\cc\left(\int_0^t \T_k(s) \left[\int_0^\tau \sum_{p=0}^{n_2} \T_p(r)u \d r\right]\d s\right) \leq \sum_{k=0}^{2n_2} \cc\left(\int_0^t \d s\int_0^\tau \T_k(s+r)u \d r\right)\\
\leq  \sum_{k=0}^{2n_2}\cc\left(\int_0^t \T_k(s) \left[\int_0^\tau \sum_{p=0}^{2n_3} \T_p(r)u \d r\right]\d s\right) \qquad \forall u \in \X_+.
\end{multline*}
Letting first $n_3$ then $n_2$ and finally $n_1$ go to infinity, we get
\begin{multline*}
\sum_{k=0}^{\infty}\cc\left(\int_0^t \T_k(s) \left[\int_0^\tau \T(r)u \d r\right]\d s\right) \leq \sum_{k=0}^{\infty} \cc\left(\int_0^t \d s\int_0^\tau \T_k(s+r)u \d r\right)\\
\leq  \sum_{k=0}^{\infty}\cc\left(\int_0^t \T_k(s) \left[\int_0^\tau  \T(r)u \d r\right]\d s\right)
\end{multline*}
i.e.
\begin{equation*}\label{lui1b}
\sum_{k=0}^{\infty}\cc\left(\int_0^t \T_k(s) \left[\int_0^\tau \T(r)u \d r\right]\d s\right) =\sum_{k=0}^{\infty} \cc\left(\int_0^t \d s\int_0^\tau \T_k(s+r)u \d r\right), \quad \forall u\in \X_+.
\end{equation*}
In particular, for any $t,\tau >0$
\begin{equation}\label{lui1b}\sum_{k=0}^{\infty}\cc\left(\int_0^t \T_k(s) \left[\int_0^\tau \T(r)u \d r\right]\d s\right) =\sum_{k=0}^{\infty} \cc\left(\int_0^\tau \T_k(s) \left[\int_0^t \T(r)u \d r\right]\d s\right).\end{equation}
From Eq. \eqref{lui1}
$$\left|\sum_{k=0}^\infty\cc\left(\int_0^t \T_k(s)u \d s\right)\right| \leq 2M \|u\| \qquad  \forall u\in \X.$$
Since $\lim_{\tau \to 0^+}\tau^{-1}\int_0^\tau \T(s)u\d s=u$, one gets that
\begin{equation}\label{luisa3}\lim_{\tau \to 0^+}\dfrac{1}{\tau} \sum_{k=0}^\infty\cc\left(\int_0^t \T_k(s)\left[\int_0^\tau \T(r)u \d r\right]\d s\right) =\sum_{k=0}^\infty\cc\left(\int_0^t \T_k(s)u \d s\right) \qquad \forall t >0.\end{equation}
Now, for $u  \in  \D(\G)_+$, Eq. \eqref{lui1} reads
\begin{multline}\label{luisa4}\sum_{k=0}^\infty\cc\left(\int_0^t \T_k(s)u \d s\right) \leq  -\la \mathbf{\Psi}, \int_0^t \T (s)\G u \d s\ra=
\left\|\int_0^t \T (s)\G u \d s \right\| \leq t \|\G u\|\end{multline}
since $\|\mathbf{\Psi}\| \leq 1.$ One extends this estimate to $\D(\G)$ in the following way: let $u \in \D(\G)$ be given and let $v=u-\G u \in \X$. Then, there exist $v_1,v_2$ in $\X_+$ with $v=v_1-v_2$ and $\|v_i\| \leq M\|v\|$, $i=1,2$. Set $u_i=(1-\G)^{-1}v_i$, $i=1,2$. Then, $u_i \in \D(\G)_+$, $\|u_i\| \leq \|v_i\|$, $i=1,2$  and $$\|\G u_1\|+\|\G u_2\|\leq 2\left(\|v_1\|+\|v_2\|\right)\leq 4M\|v\|\leq 4M\left(\|u\|+\|\G u\|\right).$$
Now, from \eqref{luisa4},
$$\left|\sum_{k=0}^\infty\cc\left(\int_0^t \T_k(s)u \d s\right)\right|=\left|\sum_{k=0}^\infty\cc\left(\int_0^t \T_k(s)(u_1-u_2) \d s\right)\right| \leq t\left(\|\G u_1\|+\|\G u_2\|\right) $$
i.e.
\begin{equation}\label{luisa5}
\left|\sum_{k=0}^\infty\cc\left(\int_0^t \T_k(s)u \d s\right)\right| \leq  4Mt\left(\|u\|+\|\G u\|\right) \qquad \forall u \in \D(\G), \:t >0.
\end{equation}
For any $v \in \X$ and any $t_1,t_2 >0$ fixed, applying the above estimate  to $u=\frac{1}{t_1}\int_0^{t_1}\T(r)v\d r-\frac{1}{t_2}\int_0^{t_2}\T(r)v\d r \in \D(\G)$ we get
{\begin{multline*}
\left|\sum_{k=0}^\infty\cc\left(\int_0^t \T_k(s) \left[\frac{1}{t_1}\int_0^{t_1}\T(r)v\d r-\frac{1}{t_2}\int_0^{t_2}\T(r)v\d r\right] \d s\right)\right| \leq \\
4Mt \left\|\frac{1}{t_1}\int_0^{t_1}\T(r)v\d r-\frac{1}{t_2}\int_0^{t_2}\T(r)v\d r\right\| + \\
4Mt \left\|\frac{1}{t_1}\G\int_0^{t_1}\T(r)v\d r-\frac{1}{t_2}\G\int_0^{t_2}\T(r)v\d r\right\|
\end{multline*}}
which, by virtue of \eqref{lui1b}, reads
{\small \begin{multline*}
\left|\frac{1}{t_1}\sum_{k=0}^\infty\cc\left(\int_0^{t_1} \T_k(s)z\d s\right)-\frac{1}{t_2}\sum_{k=0}^\infty\cc\left(\int_0^{t_2} \T_k(s)z   \d s\right)\right| \leq \\
4M  \left\|\frac{1}{t_1}\int_0^{t_1}\T(r)v\d r-\frac{1}{t_2}\int_0^{t_2}\T(r)v\d r\right\| +
4M  \left\|\frac{1}{t_1}\G\int_0^{t_1}\T(r)v\d r-\frac{1}{t_2}\G\int_0^{t_2}\T(r)v\d r\right\|
\end{multline*}}
where $z=t^{-1}\int_0^t \T(r)v\d r$. Letting now $t \to 0^+$ one deduces from \eqref{luisa3} that
{\small \begin{multline*}
\left|\frac{1}{t_1}\sum_{k=0}^\infty\cc\left(\int_0^{t_1} \T_k(s)v\d s\right)-\frac{1}{t_2}\sum_{k=0}^\infty\cc\left(\int_0^{t_2} \T_k(s)v   \d s\right)\right| \leq \\
4M  \left\|\frac{1}{t_1}\int_0^{t_1}\T(r)v\d r-\frac{1}{t_2}\int_0^{t_2}\T(r)v\d r\right\| +
4M  \left\|\frac{1}{t_1}\G\int_0^{t_1}\T(r)v\d r-\frac{1}{t_2}\G\int_0^{t_2}\T(r)v\d r\right\|
\end{multline*}}
If $v \in \D(\G)$ then $\frac{1}{t_i}\G \int_0^{t_i}\T(r)  v\d r=\frac{1}{t_i} \int_0^{t_i}\T(r)\G v\d r$, $i=1,2$ and it is easy to see that, for any $\varepsilon>0$, there exists $\delta >0$ such that
$$\left|\frac{1}{t_1}\sum_{k=0}^\infty\cc\left(\int_0^{t_1} \T_k(s)v\d s\right)-\frac{1}{t_2}\sum_{k=0}^\infty\cc\left(\int_0^{t_2} \T_k(s)v   \d s\right)\right| \leq 4M\varepsilon, \qquad \forall 0<t_1<t_2<\delta.$$
This achieves to prove that, for any $v \in \D(\G)$, the limit $\lim_{t \to 0^+} \frac{1}{t}\sum_{k=0}^\infty\cc\left(\int_0^{t} \T_k(s)v\d s\right)$ exists. We denote this limit by $\widehat{\cc}(v)$ and the first part of the Theorem is proved. The first estimate $|\widehat{\cc}(v)| \leq 4M\left(\|v\|+\|\G v\|\right)$ is a direct consequence of \eqref{luisa5}. Finally, since
$$\sum_{k=0}^\infty\cc\left(\int_0^t \T_k(s)v \d s\right) \leq \cc_0\left(\int_0^t \T(s)v \d s\right)$$
one gets that
$$\widehat{\cc}(v)=\lim_{t \to 0^+}{t}^{-1}\sum_{k=0}^\infty\cc\left(\int_0^t \T_k(s)v \d s\right)
\leq \lim_{t \to 0^+}{t}^{-1}\cc_0\left(\int_0^t \T(s)v \d s\right)=\cc_0(v)$$
since $\lim_{t \to 0^+}\frac{1}{t}\int_0^t \T(s)v \d s=v$ in the graph norm of $\G$ and $\cc_0(\cdot)$ is continuous with respect to the graph norm of $\D(\G)$. The fact that $\cc_0(v)\leq \|\G v\|$ is a direct consequence of the estimate $\|\mathbf{\Psi}\| \leq 1.$
\end{proof}

Before investigating further properties of the functional $\hat{\cc}$ we need to establish several properties of the various terms $\T_n(t)$ appearing in \eqref{dysonphi2}.

\subsection{Further properties of $\widehat{\cc}$} We are now in position to establish very useful properties of the functional $\widehat{\cc}$ complementing Proposition \ref{hatccdef}.
\begin{propo}\label{cccc0} The functional $\widehat{\cc}(\cdot) \::\;\D(\G) \to \mathbb{R}$ defined by \eqref{hatcc} is such that $\widehat{\cc}(v)=\cc_0(v)$ for any $v \in \D(\A)$. Consequently,
$$\widehat{\cc}(u)=\cc_0(u), \qquad\qquad \forall u \in \D(\overline{\A+\B}).$$
\end{propo}
\begin{proof} From \eqref{form3} one sees that, for any $n \geq 1$ and any $u \in \X_+$
\begin{equation}\label{eqq0}
\sum_{k=0}^n (\A+\B)\int_0^t \T_k(s)u \d s=\sum_{k=0}^n  \T_k(s)u - u + \B \int_0^t \T_n(s)u \d s, \qquad \forall n \in \mathbb{N}.
\end{equation}
In particular,
\begin{equation*}-\la \mathbf{\Psi}, \sum_{k=0}^n (\A+\B)\int_0^t \T_k(s)u \d s \ra = \la \mathbf{\Psi}, u \ra - \sum_{k=0}^n \la \mathbf{\Psi}, \T_k(s)u \ra - \la \mathbf{\Psi}, \B \int_0^t \T_n(s)u \d s \ra.\end{equation*}
Letting $n$ go to infinity, we see that $\lim_{n \to \infty} \left\|\B \int_0^t \T_n(s)u \d s \right\|$ exists and
\begin{equation}\label{eqq1}
\sum_{k=0}^\infty\cc\left(\int_0^t \T_k(s)u \d s\right) =\la \mathbf{\Psi}, u -\T(t)u\ra -\lim_{n \to \infty}  \la \mathbf{\Psi}, \B \int_0^t \T_n(s)u \d s \ra, \qquad \forall t >0.\end{equation}
Now, for any $u \in \D(\A)_+$ and any $k \geq 1$, one deduces from Proposition \ref{luisa2}, \textit{(2)} that
$$\la \mathbf{\Psi}, \B\T_k(s)u\ra \leq -\la \mathbf{\Psi},\A\T_k(s)u\ra=\la \mathbf{\Psi},\B\T_{k-1}(s)u\ra -\la \mathbf{\Psi},\T_k(s)\A u\ra-\la \mathbf{\Psi},\T_{k-1}(s)\B u\ra$$
and
\begin{equation}\label{eqq3}\la \mathbf{\Psi}, \B\T_k(s)u\ra -\la \mathbf{\Psi},\B\T_{k-1}(s)u\ra  \leq -\la \mathbf{\Psi},\T_k(s)\A u\ra \qquad \forall s \geq 0.\end{equation}
Since, for any $u \in \D(\A)$ the series $\sum_{k=0}^\infty\T_k(t)\A u$ converges to $\T(t)\A u$ uniformly on every bounded time interval, for any $T>0$  and any $\varepsilon >0$, there exists $N \geq 1$ such that, for any $s \in (0,T)$ and any $n \geq N$, $\left|\sum_{k=N}^n\la \mathbf{\Psi},\T_k(s)\A u\ra\right| \leq \varepsilon$. From \eqref{eqq3}, one gets
$$\sum_{k=N}^n \left(\la \mathbf{\Psi}, \B\T_k(s)u\ra -\la \mathbf{\Psi},\B\T_{k-1}(s)u\ra\right) \leq \varepsilon, \qquad \forall s \in (0,T)$$
i.e. $\la \mathbf{\Psi}, \B\T_n(s)u\ra \leq \la \mathbf{\Psi},\B\T_{N-1}(s)u\ra + \varepsilon$ for any $s \in (0,T)$.
Fixed $N >1$ and $u \in \D(\A)$, the mapping $s \in (0,T)  \mapsto  \B\T_{N-1}(s)u$ being continuous and converging to zero as $s \to 0^+$, there exists $t >0$ such that $\la \mathbf{\Psi},\B\T_{N-1}(s)u\ra < \varepsilon$ for any $0 < s < t$ and consequently $\la \mathbf{\Psi}, \B\T_n(s)u\ra < 2\varepsilon$ for any $n >N$ and any $0<s<t.$ Now, from Eq. \eqref{mustrick2} one has
$$\la \mathbf{\Psi}, \B \int_0^t \T_n(s)u \d s \ra=\la \mathbf{\Psi},  \int_0^t \B \T_n(s)u \d s \ra \leq 2\varepsilon t \qquad \forall n >N.$$
Then, one deduces from \eqref{eqq1} that
$$\left|\sum_{k=0}^\infty{t}^{-1}\cc\left(\int_0^t \T_k(s)u \d s\right) -\la \mathbf{\Psi}, \dfrac{u -\T(t)u}{t}\ra \right| \leq 2\varepsilon \qquad \forall t >0.$$
Letting $t \to 0^+$, since $\lim_{t \to 0^+}t^{-1}\la \mathbf{\Psi}, \dfrac{u -\T(t)u}{t}\ra=-\la \mathbf{\Psi}, \G u\ra=\cc_0(u)$, we get that
 $|\widehat{\cc}(u)-\cc_0(u)| \leq 2\varepsilon.$  This proves that $\widehat{\cc}$ coincides with $\cc_0$ on $\D(\A)$ since $\varepsilon$ is arbitrary. Finally, if $u \in \D(\overline{\A+\B})$, there exists a sequence $(u_n)_n \subset \D(\A)$ with $u_n \to u$ and $(\A+\B)u_n \to \G u$ as $n \to \infty$. Since $\widehat{\cc}(u_n)=\cc_0(u_n)$ for any $n\in \mathbb{N}$, one deduces easily that $\widehat{\cc}(u)=\cc_0(u).$\end{proof}

\subsection{Mild honesty} We introduce now another concept of honest trajectories. To distinguish it \
\textit{a priori} from the previous one, we will speak rather of mild
honesty.
\begin{defi} Let $u \in \X_+$ be given. Then, the trajectory
$(\T(t)u)_{t \geq 0}$ is said to be \textit{\textbf{mild honest}} if and
only if
$$\|\T(t)u\|=\| u\|-\widehat{\cc}\bigg(\int_0^t \T(s)u \d
s\bigg),\quad \text{ for any }   t\geq 0.$$
\end{defi}
We are now in position to state the main result of this section, reminiscent to Theorem \ref{equivalence}:
\begin{theo}\label{equivalencemild} Given $u \in \X_+$, the following statements are equivalent
\begin{enumerate}
\item the trajectory $(\T(t)u)_{t \geq 0}$ is mild honest;
\item $\lim_{n \to \infty}\|\B\int_0^t \T_n(s)u\d s\|=0$ for any $t >0$;
\item $\int_0^t \T(s)u\d s \in \D(\overline{\A+\B})$ for any $t >0$;
\item the set $\left(\B\int_0^t \T_n(s)u\d s\right)_n$ is relatively weakly compact in $\X$ for any $t >0$.
\end{enumerate}
\end{theo}
\begin{proof} Let $u \in \X_+$ and $t>0$ be fixed. One has $\int_0^t \T(s)u\d s \in \D(\G)$ and
$$\widehat{\cc}\left(\int_0^t \T(s)u\d s\right)=\lim_{\tau \to 0^+}\tau^{-1}\sum_{n=0}^\infty \cc\left(\int_0^\tau \T_n(s)\d s\left[\int_0^t \T(r)u \d r\right]\right).$$
From \eqref{lui1b} and \eqref{luisa3}, it is easy to deduce that
\begin{equation}\label{egaliteX}\widehat{\cc}\left(\int_0^t \T(s)u\d s\right)=\sum_{n=0}^\infty \cc\left(\int_0^t \T_n(s)u\d s\right), \qquad \forall u \in \X_+, \:t >0.\end{equation}
Thus, Eq. \eqref{eqq1} can be rewritten as
$$\widehat{\cc}\left(\int_0^t \T(s)u\d s\right)=\la \mathbf{\Psi},u-\T(t)u\ra - \lim_{n \to \infty}\left\|\B\int_0^t \T_n(s)u\d s\right\|.$$
This proves immediately that $\textit{(1)} \Longleftrightarrow \textit{(2)}.$ Let us prove that $\textit{(2)} \Longrightarrow \textit{(3)}$. Observe that, according to \eqref{eqq0}
$$\left(\A+\B\right)\left(\sum_{k=0}^n \int_0^t \T_k(s)u\d s\right)=\sum_{k=0}^n \T_k(t)u - u + \B\int_0^t \T_n(s)u\d s$$
so that, from \textit{(2)} we deduce that the right-hand side converges to $\T(t)u-u$ as $n$ goes to infinity. Since $\sum_{k=0}^n \int_0^t \T_k(s)u\d s$ converges to $\int_0^t\T(s)u\d s$ as $n$ goes to infinity, one gets immediately that \textit{(3)} holds with $\left(\overline{\A+\B}\right)\int_0^t\T(s)u\d s=\T(t)u-u.$ Let us now assume that $\textit{(3)}$ holds. Then, from \eqref{cccc0},
$$\widehat{\cc}\left(\int_0^t \T(s)u\d s\right)=\cc_0\left(\int_0^t \T(s)u\d s\right)$$
i.e. $\widehat{\cc}\left(\int_0^t \T(s)u\d s\right)=\|u\|-\|\T(t)u\|$ which is nothing but \textit{(1)}. Assume now \textit{(4)} to hold. Then, up to extracting a subsequence, we may assume that $\B\int_0^t \T_n(s)u\d s$ converges weakly to some $v \in \X.$ Then, $\sum_{k=0}^n \int_0^t\T_k(s)u\d s$ converges weakly to $\int_0^t\T(s)u\d s$ while
$$\left(\A+\B\right)\sum_{k=0}^n \int_0^t\T_k(s)u\d s \textrm{ converges weakly to } \left(\T(t)u-u - v\right).$$
In particular, $\left(\int_0^t \T(s)u\d s, \T(t)u-u - v\right)$ belongs to the weak closure (and thus the strong closure) of the graph of $\A+\B$. In particular, \textit{(3)} holds. Finally, it is clear that $\textit{(2)} \implies \textit{(4)}.$
\end{proof}
The following result proves that the two notions of honesty and mild honesty are equivalent:
\begin{theo}\label{final} The two functionals $\widehat{a}$ and $\overline{a}$ coincide and
consequently the notions of honest or mild honest trajectories are
equivalent.
\end{theo}
\begin{proof} Let $u \in \X_+$ and $\l >0$ be given.  One deduces from \eqref{egaliteX} that
$$\int_0^\infty \exp(-\l t)\widehat{\cc}\left(\int_0^t \T(s)u\d s\right)\d t=\sum_{k=0}^\infty \int_0^\infty \exp(-\l t){\cc}\left(\int_0^t \T_k(s)u\d s\right)\d t $$
because all the functions involved are positive. On the other hand, since the mapping $t \geq 0 \mapsto \int_0^t \T(s)u\d s \in \D(\G)$ is continuous as well as the mapping $t \geq 0 \mapsto \int_0^t \T_k(s)u\d s \in \D(\A) \subset \D(\G)$, we have
\begin{equation*}\begin{split}
\int_0^\infty \exp(-\l t)\widehat{\cc}\left(\int_0^t \T(s)u\d s\right)\d t&=\widehat{\cc}\left(\int_0^\infty \exp(-\l t)\d t \int_0^t \T(s)u\d s\right)\\
&=\dfrac{1}{\l}\widehat{\cc}\left(\int_0^\infty \exp(-\l s)\T(s)u\d s\right)=\dfrac{1}{\l}\widehat{\cc}\left((\l-\A)^{-1}u\right)
\end{split}
\end{equation*}
since $\widehat{\cc}$ is $\D(\G)$-continuous. We also have, for any $k \in \mathbb{N}$
\begin{equation*}\begin{split}
\int_0^\infty &\exp(-\l t) {\cc}\left(\int_0^t \T_k(s)u\d s\right)\d t = {\cc_0}\left(\int_0^\infty \exp(-\l t)\d t \int_0^t \T_k(s)u\d s\right)\\
&=\dfrac{1}{\l} {\cc_0}\left(\int_0^\infty \exp(-\l s)\T_k(s)u\d s\right) =\dfrac{1}{\l} {\cc_0}\left((\l-\A)^{-1}\left(\B(\l-\A)^{-1}\right)^ku\right) \\& =\dfrac{1}{\l} {\cc}\left((\l-\A)^{-1}\left(\B(\l-\A)^{-1}\right)^ku\right)
\end{split}
\end{equation*}
where we used \eqref{LnTn} and the fact that $\cc_0$ is $\D(\G)$-continuous. Hence
$$\widehat{\cc}\left((\l-\A)^{-1}u\right)=\sum_{k=0}^\infty {\cc}\left((\l-\A)^{-1}\left(\B(\l-\A)^{-1}\right)^ku\right)$$
which proves (see Subsection 3.1) that $\widehat{\cc}=\overline{\cc}.$
\end{proof}
\begin{nb} The above provides an alternative proof of Proposition \ref{prop1.1.mkvo}.
\end{nb}
\begin{nb} The equivalence between the two notions of honesty established here above has some unsuspected consequences. For instance, one notes that Theorems \ref{equivalence} and \ref{equivalencemild} imply that, for a given $u \in \X_+$,
$$\int_0^t \T(s)u\d s \in \D(\overline{\A+\B}) \quad \forall t >0 \Longleftrightarrow \left[\B(\l-\A)^{-1}u\right]^n \to 0 \text{ as } n \to \infty.$$
Notice also that, for any $u \in \X_+$, the mass loss functional $\Delta_u(t)$ defined in Section \ref{sub:instant} is given by $\Delta_u(t)=\lim_{n \to \infty}
\displaystyle \left\|\B\int_0^t\T_n(s) u \d
s\right\|$.
\end{nb}

\end{document}